\documentclass[a4paper]{article}
\usepackage{amsmath,amsthm,amssymb,amscd,graphicx, enumerate}

\numberwithin{equation}{section}

\newtheorem{thm}{Theorem}[section]
\newtheorem{lem}[thm]{Lemma}
\newtheorem{prop}[thm]{Proposition}
\newtheorem{cor}[thm]{Corollary}
\newtheorem{remark}[thm]{Remark}

\theoremstyle{definition}
\newtheorem{definition}[thm]{Definition}
\newtheorem{example}[thm]{Example}

\newenvironment{rem}{\begin{remark}\rm}{\end{remark}}

\makeatother

\title{Stabilities of affine Legendrian submanifolds and their moduli spaces}
\author{Kotaro Kawai
\footnote{The author is supported by Grant-in-Aid for JSPS fellows (26-7067).}}
\date{}

\begin{document}

\maketitle

\begin{abstract}
We introduce the notion of 
affine Legendrian submanifolds 
in Sasakian manifolds 
and define a canonical volume called the $\phi$-volume 
as 
odd dimensional analogues of affine Lagrangian (totally real or purely real) geometry. 
Then 
we derive the second 
variation formula of the $\phi$-volume 
to obtain the stability result 
in some $\eta$-Einstein Sasakian manifolds. 
It also implies the convexity of the $\phi$-volume functional 
on the space of affine Legendrian submanifolds.

Next, we introduce the notion of 
special affine Legendrian submanifolds in Sasaki-Einstein manifolds 
as a generalization of that of special Legendrian submanifolds. 
Then we show that 
the moduli space 
of compact connected special affine Legendrian submanifolds 
is a smooth Fr\'{e}chet manifold. 
\end{abstract}

\section{Introduction}

Affine Lagrangian (totally real or purely real) 
submanifolds are 
``maximally non-complex" submanifolds 
in almost complex manifolds 
defined by relaxing the Lagrangian condition (Definition \ref{def of aff Lag}). 
The affine Lagrangian condition is an open condition 
and hence there are many examples. 
Borrelli \cite{Borrelli} 
defined a canonical volume of an affine Lagrangian submanifold
called the $J$-volume. 
He obtained the stability result for the $J$-volume 
as in the Lagrangian case \cite{Chen}. 
Lotay and Pacini \cite{LotayPacini} pointed out the importance of 
affine Lagrangian submanifolds 
in the study of geometric flows. 
Opozda \cite{Opozda} showed that 
the moduli space of (special) affine Lagrangian submanifolds 
was a smooth Fr\'{e}chet manifold.

In this paper, 
we study the odd dimensional analogue. 
First, we introduce the notion of 
affine Legendrian submanifolds 
in Sasakian manifolds 
and 
define a canonical volume called the $\phi$-volume 
as 
odd dimensional analogues of affine Lagrangian geometry. 
See Definitions \ref{def of aff Leg} and \ref{def of phi vol}. 
Then we compute the first variation of the $\phi$-volume  
and characterize 
a critical point for the $\phi$-volume 
by the vanishing of some vector field $H_{\phi}$ (Proposition \ref{1st var phi Vol}), 
which is a generalization of the mean curvature vector (Remark \ref{MC in Leg case}). 
We call an affine Legendrian submanifold {\bf $\phi$-minimal} if $H_{\phi} = 0$. 
Then we compute the second variation of the $\phi$-volume 
and obtain the following.

\begin{thm} \label{2nd var phi Vol}
Let $(M^{2n+1}, g, \eta, \xi, \phi)$ be a 
$(2n+1)$-dimensional Sasakian manifold 
and $\iota : L^{n} \hookrightarrow M$ 
be an affine Legendrian immersion 
of a compact oriented $n$-dimensional manifold $L$. 
Let  
$\iota_{t} : L \hookrightarrow M$ 
be a one-parameter family of affine Legendrian immersions 
satisfying $\iota_{0} = \iota$. 
Suppose that 
$\frac{\partial \iota_{t}}{\partial t}|_{t=0} = Z
= \phi Y + f \xi$, 
where $Y \in \mathfrak{X}(L)$ is a vector field on $L$ 
and $f \in C^{\infty}(L)$ is a smooth function.
Then we have 
\begin{align*}
\left. \frac{d^{2}}{d t^{2}} 
\int_{L} {\rm vol}_{\phi} [\iota_{t}] \right|_{t=0} 
= &
\int_{L} 
\left(
(2n+2) \eta (Y)^{2} - 2 g(Y, Y) - {\rm Ric}(Y, Y)  \right. \\
&
\left. 
- g(\pi_{L}[Z, Y], H_{\phi}) 
+ g(Y, H_{\phi})^{2}
+ 
\left (
\frac{{\rm div}(\rho_{\phi} [\iota] Y)}{\rho_{\phi} [\iota]}
\right)^{2}
\right)
{\rm vol}_{\phi}[\iota],
\end{align*}
where 
${\rm vol}_{\phi}[\iota]$ is the $\phi$-volume form of $\iota$ given in Definition \ref{def of phi vol}, 
${\rm Ric}$ is the Ricci curvature of $(M,g)$, 
$\pi_{L}: \iota^{*}TM \rightarrow \iota_{*} TL$ 
is the canonical projection given in (\ref{canonical proj}), 
$\rho_{\phi} [\iota]$ is the function on $L$ given in Definition \ref{def of phi vol} and 
$H_{\phi}$ is the vector field on $L$ given in Definition \ref{def of H phi}. 
\end{thm}

\begin{rem}
For Legendrian submanifolds, 
the $\phi$-volume agrees with the standard Riemannian volume (Lemma \ref{equality rho phi}). 
When $\iota$ is minimal Legendrian and all of $\iota_{t}$'s are Legendrian, 
Theorem \ref{2nd var phi Vol} agrees with \cite[Theorem 1.1]{Ono}. 
When $\iota$ is Legendrian-minimal Legendrian 
and all of $\iota_{t}$'s are Legendrian, 
Theorem \ref{2nd var phi Vol} agrees with \cite[Theorem 1.1]{Kajigaya}. 
See Remark \ref{relation Lmin}. 
\end{rem}

Following the Riemannian case, 
we call a $\phi$-minimal affine Legendrian submanifold {\bf $\phi$-stable} 
if the second variation of the $\phi$-volume is nonnegative.

Now, suppose that 
a $(2n+1)$-dimensional Sasakian manifold $(M^{2n+1}, g, \eta, \xi, \phi)$ 
is a $\eta$-Einstein with the 
$\eta$-Ricci constant $A \in \mathbb{R}$. 
(See Definition \ref{def of eta Einstein}.)
Then we obtain the following.

\begin{thm} \label{thm stablity}
Let $(M^{2n+1}, g, \eta, \xi, \phi)$ be a $(2n+1)$-dimensional 
$\eta$-Einstein Sasakian manifold with the 
$\eta$-Ricci constant $A \leq -2$. 
Then any $\phi$-minimal affine Legendrian submanifold in $M$ is 
$\phi$-stable. 
\end{thm}

This is a generalization of \cite[Theorem 1.2]{Ono}. 
The author obtained further results 
by restricting variations of a minimal Legendrian submanifold to its Legendrian variations. 
In our case, since the affine Legendrian condition is an open condition, 
any small variation is affine Legendrian. 
Thus there is no analogue of these results.

Similarly, using the notion of convexity 
in the space of affine Legendrian submanifolds (Definition \ref{def of convex}), 
we easily see the following. 

\begin{thm} \label{thm convexity}
Let $(M^{2n+1}, g, \eta, \xi, \phi)$ be a $(2n+1)$-dimensional 
$\eta$-Einstein Sasakian manifold with the 
$\eta$-Ricci constant $A \leq -2$. 
Then the $\phi$-volume functional 
on the space of affine Legendrian submanifolds is convex. 
\end{thm}

For affine Legendrian submanifolds in a $\eta$-Einstein Sasakian manifold with the 
$\eta$-Ricci constant $A > -2$, 
we have the following.

\begin{thm} \label{thm obstruction}
Let $(M^{2n+1}, g, \eta, \xi, \phi)$ be a $(2n+1)$-dimensional 
$\eta$-Einstein Sasakian manifold with the 
$\eta$-Ricci constant $A > -2$. 
Then there are no 
$\phi$-minimal affine Legendrian submanifolds 
which are $\phi$-stable.
\end{thm}

Next, 
we define a special affine Legendrian submanifold in 
a Sasaki-Einstein manifold 
with a Calabi-Yau structure on its cone
by requiring that its cone is special affine Lagrangian 
(Definition \ref{def of special aff Leg}). 
This notion is a generalization of that of special Legendrian submanifolds. 
By a slight generalization of the general deformation theory of 
Moriyama \cite[Proposition 2.2]{Moriyama}, 
we 
study the moduli space of special affine Legendrian submanifolds
and obtain the following.

\begin{thm} \label{smooth moduli affine Leg}
Let $M$ be a Sasakian manifold with 
a Calabi-Yau structure on its cone 
and $L$ be a compact connected manifold 
admitting a special affine Legendrian embedding $L \hookrightarrow M$. 
Then 
the moduli space 
of special affine Legendrian embeddings of $L$ 
is an infinite dimensional smooth Fr\'{e}chet manifold 
modeled on the Fr\'{e}chet vector space 
$\{ (g, \alpha) \in C^{\infty}(L) \oplus \Omega^{1}(L); (n+1)g + d^{*} \alpha = 0 \} \cong \Omega^{1}(L)$, 
which is identified with 
the direct sum of 
the space of functions with integral $0$ and that of coclosed 1-forms. 
It is a submanifold of the 
moduli space of smooth affine Legendrian embeddings of $L$.
\end{thm}

\begin{rem}
Theorem \ref{smooth moduli affine Leg} shows 
the different property of 
the moduli space of special affine Legendrian submanifolds 
from 
that of special Legendrian submanifolds. 
In general, 
there are obstructions of special Legendrian deformations. 
See \cite[Section 4.2]{Moriyama}.
\end{rem}

This paper is organized as follows. 
In Section 2, we review the fundamental
facts of Sasakian geometry. 
In Section 3, 
we review affine Lagrangian geometry 
and introduce its odd dimensional analogue, 
namely, 
affine Legendrian geometry. 
In Section 4, we compute the first variation of the $\phi$-volume.
In Section 5, 
we compute the second variation of the $\phi$-volume 
to obtain Theorems \ref{2nd var phi Vol},  
\ref{thm stablity}, \ref{thm convexity} and \ref{thm obstruction}. 
In Section 6, 
we consider the $\phi$-volume in Sasaki-Einstein manifolds 
and introduce the notion of 
special affine Legendrian submanifolds. 
In Section 7, 
we study the moduli space of special affine Legendrian submanifolds 
and prove Theorem \ref{smooth moduli affine Leg}.

\noindent{{\bf Acknowledgements}}: 
The author owes a great intellectual debt to 
the work of Lotay and Pacini \cite{LotayPacini}, 
which motivates him to study the subject of this paper. 
He is grateful to Professor Takayuki Moriyama 
for letting me know his work and 
fruitful discussions. 
He also thanks the referee for 
useful comments on an earlier version of this paper.


\section{Sasakian geometry}

\begin{definition} \label{def of Sasakian 1}
Let $M^{2n + 1}$ be a $(2n + 1)$-dimensional manifold. 
Suppose that  
there exist 
a contact form $\eta$, 
a Riemannian metric $g$, a Killing vector field $\xi$
and a type (1, 1)-tensor $\phi$ on $M$. 
We call 
$(M, g, \eta, \xi, \phi)$
a {\bf Sasakian manifold} if 
we have 
\begin{align*}
\eta (\xi) &= 1,  \\
\phi^{2}   &= -id_{TM}+ \eta \otimes \xi, \\
g (\phi (X), \phi (Y)) &= g(X,Y) -\eta(X) \eta(Y), \\
d \eta                  &= 2 g (\cdot, \phi (\cdot) ), \\
[\phi,\phi](X, Y)+ d \eta (X, Y) \xi &= 0, 
\end{align*}
where 
$X, Y \in TM$, 
$d \eta (X,Y) = X \eta(Y) - Y \eta(X) - \eta([X, Y])$ and 
$[\phi, \phi](X, Y ) = 
\phi^{2} [X, Y] + [\phi(X),\phi(Y)] -
\phi [\phi(X),Y]-\phi [X, \phi(Y )]$. 
\end{definition}

We can also define a Sasakian manifold 
in terms of a Riemannian cone.

\begin{definition} \label{def of Sasakian 2}
An odd dimensional Riemannian manifold 
$(M, g)$ is a {\bf Sasakian manifold} if 
its Riemannian cone 
$(C(M), \bar{g}) = (\mathbb{R}_{>0} \times M, dr^{2} + r^{2} g)$ 
is a K\"ahler manifold with respect to some complex structure $J$ over $C(M)$. 
\end{definition}
Here, $r$ is a standard coordinate of $\mathbb{R}_{>0}$ and we regard $r$ as
the function on $C(M)$. 
We identify $M$ with the submanifold $\{1\} \times M \subset C(M)$. 

It is known that 
Definitions \ref{def of Sasakian 1} and \ref{def of Sasakian 2} are equivalent. 
From Definition \ref{def of Sasakian 2}, we see that 
Sasakian geometry is regarded as the odd-dimensional analogue of K\"{a}hler geometry.

Tensors in Definition \ref{def of Sasakian 1} 
are recovered as follows. 
Define the vector field $\tilde{\xi}$ and the 1-form $\tilde{\eta}$
on $C(M)$ by 
\begin{align*}
\tilde{\xi}  = - J \left( r\frac{\partial}{\partial r} \right), \qquad 
\tilde{\eta} = \frac{1}{r^{2}}\bar{g}(\tilde{\xi},\cdot)=\frac{dr \circ J}{r}=-2d^{c}_{J}{\rm log}r, 
\end{align*}
where $d^{c}_{J} f = -df \circ J/2 = i (\bar{\partial} - \partial)f/2$ for the function $f$ on $C(M)$.
Then we have 
\begin{align*}
\xi = \tilde{\xi}|_{r=1}, \qquad
\eta= \tilde{\eta}|_{r=1}.
\end{align*}
By the decomposition 
$TM = {\rm ker}\eta \oplus \mathbb{R}\xi$, 
the endomorphism $\phi \in C^{\infty}(M, {\rm End}(TM))$ is given by 
\begin{align*}
\phi|_{{\rm ker}\eta}=J|_{{\rm ker}\eta}, \qquad
\phi|_{\mathbb{R}\xi}=0. 
\end{align*}
The metric $g$ on $M$, $J|_{ \{ r=1 \} }$
and the K\"{a}hler form $\bar{\omega}$ of $\bar{g}$
are described as 
\begin{align*}
g         &= \frac{1}{2}d\eta (\phi(\cdot), \cdot)+\eta \otimes \eta, \\
J|_{ \{ r=1 \} }&= \phi - \xi \otimes dr + \frac{\partial}{\partial r} \otimes \eta, \\
\bar{\omega}&= \frac{1}{2}d(r^{2}\tilde{\eta})=-\frac{1}{2}dd^{c}_{J}r^{2}.
\end{align*}

We summarize basic equations in Sasakian geometry. 
See \cite[Section 2]{Ono}.

\begin{lem} \label{Sasaki eq1}
Let $(M, g, \eta, \xi, \phi)$ be a (2n+1)-dimensional Sasakian manifold. 
Then we have
\begin{align*}
\phi (\xi) =0, \qquad
\eta \circ \phi =0, \qquad
\eta = g(\xi, \cdot), \qquad
i (\xi) d \eta = 0, 
\end{align*}
where $i(\cdot )$ is an inner product. 
\end{lem}

\begin{lem} \label{Sasaki eq2}
Let $(M, g, \eta, \xi, \phi)$ be a Sasakian manifold. 
Then we have
\begin{align*}
\nabla_{X} \xi =& - \phi (X), \\
(\nabla_{X} \phi) (Y) =& g(X, Y) \xi - \eta(Y) X, \\
R(X, Y)\xi =& \eta (Y) X - \eta(X) Y, \\
R(X, Y) (\phi (Z))
=&
\phi (R(X, Y)Z) - g(Y, Z) \phi(X) 
+g(\phi(X), Z) Y \\
&+ g(X, Z) \phi (Y)
- g(\phi (Y), Z) X, \\
{\rm Ric} (\xi, X) 
=& 
\left\{ \begin{array}{ll}
2n & (X=\xi) \\
0 & (X \in \ker \eta) \\
\end{array} \right.
\end{align*}
where $X, Y, Z \in \mathfrak{X}(M)$ are vector fields on $M$,  
$\nabla$ is the Levi-Civita connection of $(M, g)$, 
$R$ is the curvature tensor of $(M, g)$ 
and 
${\rm Ric}$ is the Ricci curvature tensor of $(M, g)$.
\end{lem}

Note that when $(M, g)$ is Einstein, the scalar curvature is equal to $2n(2n+1)$.

\begin{definition} \label{def of eta Einstein}
A $(2n+1)$-dimensional Sasakian manifold $(M^{2n+1}, g, \eta, \xi, \phi)$ 
is called {\bf $\eta$-Einstein} if we have 
\begin{align*}
{\rm Ric} = A g + (2n-A) \eta \otimes \eta 
\end{align*}
for some $A \in \mathbb{R}$. 
The constant 
$A$ is called the {\bf $\eta$-Ricci constant}. 
\end{definition}

This condition is necessary to prove Theorems \ref{thm stablity} and \ref{thm convexity}.
Note that $g$ is Einstein if $A=2n$.

Let $L$ be an $n$-dimensional manifold admitting an immersion
$\iota: L \hookrightarrow M$ into a $(2n+1)$-dimensional Sasakian manifold. 
The immersion $\iota$ induces the immersion 
\begin{align} \label{induced imm}
\bar{\iota}: C(L) = \mathbb{R}_{>0} \times L
\ni (r, x) \mapsto
(r, \iota(x))
\in 
\mathbb{R}_{>0} \times M = C(M).
\end{align}

\begin{definition}
An immersion $\iota: L \hookrightarrow M$ is {\bf Legendrian}
if $\iota^{*} \eta = 0$. 
This is equivalent to the condition that 
the induced immersion $\bar{\iota}:  C(L) \hookrightarrow C(M)$ given by (\ref{induced imm}) 
is Lagrangian: $\bar{\iota}^{*} \bar{\omega} = 0$. 
\end{definition}

\section{Affine Lagrangian and affine Legendrian submanifolds}

First, we review affine Lagrangian geometry by \cite{Borrelli, LotayPacini}.

\subsection{Affine Lagrangian submanifolds}

Let $(X, h, J, \omega)$ be a real $2n$-dimensional 
almost Hermitian manifold, 
where $h$ is a Hermitian metric, 
$J$ is an almost complex structure 
and $\omega$ is an associated K\"{a}hler form. 
Let $N$ be an oriented $n$-dimensional manifold 
admitting 
an immersion $f: N \hookrightarrow X$.

\begin{definition} \label{def of aff Lag}
An immersion $f$ is called {\bf affine Lagrangian} if 
\begin{align} \label{def aff Lag decomp}
T_{f (x)} X = f_{*} (T_{x}N) \oplus J f_{*} (T_{x}N)
\end{align}
for any $x \in N$.
\end{definition}

\begin{rem}
If $N$ is Lagrangian (i.e. $f^{*} \omega = 0$), (\ref{def aff Lag decomp}) 
is an orthogonal decomposition. 
The affine Lagrangian condition does not require 
the orthogonality of the decomposition (\ref{def aff Lag decomp}). 
\end{rem}

\begin{rem}
The affine Lagrangian condition is an open condition. 
The metric is not needed in the definition, 
and hence we can define affine Lagrangian submanifolds 
in an almost complex manifold. 
\end{rem}

Next, 
we define a $J$-volume 
introduced by Borrelli \cite{Borrelli}.

\begin{definition} \label{def of J vol}
Let $f: N^{n} \hookrightarrow X^{2n}$ be 
an affine Lagrangian immersion. 
Define the 
{\bf $J$-volume form} ${\rm vol}_{J} [f]$ of $f$, 
which is the $n$-form on $N$, 
by 
\begin{align*}
{\rm vol}_{J} [f] = \rho_{J} [f] {\rm vol}_{f^{*} h}, 
\end{align*}
where 
${\rm vol}_{f^{*} h}$ is the Riemannian volume form of $f^{*} h$ and 
the function $\rho_{J} [f]$ on $N$ is defined by 
\begin{align*}
\rho_{J}[f](x) = 
\sqrt
{{\rm vol}_{h}(f_{*} e_{1}, \cdots, f_{*} e_{n}, J f_{*} e_{1}, \cdots, J f_{*} e_{n})}
\end{align*}
for $x \in N$ 
and $\{e_{1},\cdots, e_{n} \}$ is an orthonormal basis of $T_{x} N$.

When $N$ is compact, define the 
{\bf $J$-volume} ${\rm Vol}_{J} [f]$ of $N$ by 
\begin{align*}
{\rm Vol}_{J} [f] = \int_{N} {\rm vol}_{J} [f].
\end{align*}
\end{definition}

\begin{rem}
The 
definition of the $J$-volume form ${\rm vol}_{J} [f]$ 
is independent of the choice of the Hermitian metric $h$.
See \cite[Section 3.2, 4.1]{LotayPacini}. 
Thus the $J$-volume is also defined in an almost complex manifold. 
\end{rem}

By definition, the following is easy to prove and we omit the proof.

\begin{lem} \label{equality rho J}
We have $0 \leq \rho_{J}[f] \leq 1$. 
The equality $\rho_{J}[f] = 1$ holds if and only if $f$ is Lagrangian. 
The equality $\rho_{J}[f] = 0$ holds if and only if $f$ is partially complex, 
namely, $f_{*} T_{x} N$ contains a complex line for any $x \in N$. 
\end{lem}

\begin{lem} \label{diff equiv aff Lag}
For any diffeomorphism $\varphi \in {\rm Diff}^{\infty}(N)$ of $N$, we have
\begin{align*}
{\rm vol}_{J} [f \circ \varphi] = \varphi^{*} {\rm vol}_{J} [f].
\end{align*}
Thus when $N$ is compact, we obtain
\begin{align*}
{\rm Vol}_{J} [f \circ \varphi] = {\rm Vol}_{J} [f].
\end{align*}
\end{lem}

\begin{proof}
We easily see that 
$\rho_{J}[f \circ \varphi] (x) = \rho_{J}[f] (\varphi (x))$, 
${\rm vol}_{(f \circ \varphi)^{*} h} = \varphi^{*} {\rm vol}_{f^{*} h}$, 
which imply the statement. 
\end{proof}

\subsection{Affine Legendrian submanifolds}

Next, we introduce the odd dimensional analogue of affine Lagrangian geometry, 
namely, affine Legendrian geometry. 
Let $(M,g, \eta, \xi, \phi)$ be a $(2n+1)$-dimensional Sasakian manifold 
and $L$ be an oriented $n$-dimensional manifold 
admitting 
an immersion $\iota: L \hookrightarrow M$.

\begin{definition} \label{def of aff Leg}
An immersion $\iota$ is called {\bf affine Legendrian} if 
\begin{align} \label{def aff Leg decomp}
T_{\iota (x)} M = \iota_{*} (T_{x}L) \oplus \phi \iota_{*} (T_{x}L) \oplus \mathbb{R} \xi_{\iota (x)}
\end{align}
for any $x \in L$.
\end{definition}

Then we can define canonical projections 
\begin{align}\label{canonical proj}
\pi_{L}:     T_{x} M \rightarrow \iota_{*} (T_{x}L), \qquad
\pi_{\phi}: T_{x} M \rightarrow \phi \iota_{*} (T_{x}L), \qquad
\pi_{\xi}:   T_{x} M \rightarrow \mathbb{R} \xi_{\iota (x)}.
\end{align}

\begin{rem}
If L is Legendrian (i.e. $\iota^{*} \eta = 0$), (\ref{def aff Leg decomp}) 
is an orthogonal decomposition. 
The affine Legendrian condition does not require 
the orthogonality of the decomposition (\ref{def aff Leg decomp}). 
\end{rem}

\begin{rem}
We can define affine Legendrian submanifolds 
in an almost contact manifold.
To simplify the computations, especially in Sections 4 and 5, 
we assume that $M$ is Sasakian. 
\end{rem}

By definition, we easily see the following.

\begin{rem}
An immersion $\iota: L \rightarrow M$ 
is affine Legendrian 
if and only if 
$\bar{\iota}: C(L) \rightarrow C(M)$ given by (\ref{induced imm})
is affine Lagrangian. 
\end{rem}

Next, 
we define the $\phi$-volume 
as an analogue of the $J$-volume. 
Recall that 
the Riemannian volume form ${\rm vol}_{\bar{\iota}^{*} \bar{g}}$ 
of $\bar{\iota}^{*} \bar{g}$ on $C(L)$
and 
the Riemannian volume form ${\rm vol}_{\iota^{*} g}$ 
of $\iota^{*} g$ on $L$
are related by 
$
{\rm vol}_{\bar{\iota}^{*} \bar{g}} = r^{n} dr \wedge {\rm vol}_{\iota^{*} g}.
$
As an analogue of this fact, 
we define the $\phi$-volume.

\begin{definition} \label{def of phi vol}
Let $\iota: L^{n} \hookrightarrow M^{2n+1}$ be 
an affine Legendrian immersion into 
a Sasakian manifold. 
Define the 
{\bf $\phi$-volume form} ${\rm vol}_{\phi} [\iota]$ of $\iota$, 
which is the $n$-form on $L$, 
by 
\begin{align*}
{\rm vol}_{J} [\bar{\iota}] = r^{n} dr \wedge {\rm vol}_{\phi} [\iota].
\end{align*}
When $L$ is compact, define the 
{\bf $\phi$-volume} ${\rm Vol}_{\phi} [\iota]$ of $L$ by 
\begin{align*}
{\rm Vol}_{\phi} [\iota] = \int_{L} {\rm vol}_{\phi} [\iota].
\end{align*}
\end{definition}

The $\phi$-volume form ${\rm vol}_{\phi} [\iota]$ is described 
more explicitly as follows. 
Define the function $\rho_{\phi}[\iota]$ on $L$ by
\begin{align*}
\rho_{\phi}[\iota](x) &= \rho_{J}[\bar{\iota}] (1, x)\\
&=
\sqrt
{{\rm vol}_{g}(\iota_{*} e_{1}, \cdots, \iota_{*} e_{n}, -\xi, \phi \iota_{*} e_{1}, \cdots, \phi \iota_{*} e_{n})} 
\end{align*}
for $x \in L$ and $\{e_{1},\cdots, e_{n} \}$ is an orthonormal basis of $T_{x} L$. 
Then we see that 
\begin{align*}
{\rm vol}_{\phi} [\iota] = \rho_{\phi} [\iota] {\rm vol}_{\iota^{*} g}.
\end{align*}

As in the affine Lagrangian case, we easily see the following.

\begin{rem}
The 
definition of the $\phi$-volume form ${\rm vol}_{\phi} [\iota]$ 
is independent of the choice of the Sasakian metric $g$.
\end{rem}

\begin{lem} \label{equality rho phi}
We have $0 \leq \rho_{\phi}[\iota] \leq 1$. 
The equality $\rho_{\phi}[\iota] = 1$ holds if and only if $\iota$ is Legendrian. 
The equality $\rho_{\phi}[\iota] = 0$ holds if and only if 
for any $x \in L$, there exists $0 \neq X \in T_{x} L$ such that 
$\iota_{*} X, \phi \iota_{*} X \in \iota_{*} T_{x} L$ 
or 
$\iota_{*} X$ or $\phi \iota_{*} X$ is a multiple of $\xi_{\iota (x)}$. 
\end{lem}

\begin{lem} \label{diff equiv aff Leg}
For any diffeomorphism $\varphi \in {\rm Diff}^{\infty}(L)$ of $L$, we have
\begin{align*}
{\rm vol}_{\phi} [\iota \circ \varphi] = \varphi^{*} {\rm vol}_{\phi} [\iota].
\end{align*}
Thus when $L$ is compact, we obtain
\begin{align*}
{\rm Vol}_{\phi} [\iota \circ \varphi] = {\rm Vol}_{\phi} [\iota].
\end{align*}
\end{lem}

\subsection{Geodesics and convexity}

In \cite[Section 3.1]{LotayPacini}, the notion of geodesics 
in the space of affine Lagrangian submanifolds 
was introduced. 
Analogously, 
we define the notion of geodesics 
in the space of affine Legendrian submanifolds. 

Let $(M, g, \eta, \xi, \phi)$ be a $(2n+1)$-dimensional Sasakian manifold and 
$L$ be an oriented $n$-dimensional manifold admitting an embedding $L \hookrightarrow M$. 
Let $\mathcal{P}$ be the space of all affine Legendrian embeddings of $L$:
\begin{align*}
\mathcal{P} = \{ \iota: L \hookrightarrow M; \iota \mbox{ is an affine Legendrian embedding} \}.
\end{align*}
The group ${\rm Diff}^{\infty}(L)$ of diffeomorphisms of $L$ 
acts freely on $\mathcal{P}$ on the right 
by reparametrizations. 
Set $\mathcal{A} = \mathcal{P}/ {\rm Diff}^{\infty}(L)$.
Thus we can regard $\mathcal{P}$ as 
a principal ${\rm Diff}^{\infty}(L)$-bundle over $\mathcal{A}$.
Denote by $\pi: \mathcal{P} \rightarrow \mathcal{A}$ the canonical projection. 

For each $\iota \in \mathcal{P}$, define the subspaces of $T_{\iota} \mathcal{P}$ by
\begin{align*}
V_{\iota} = \{ \iota_{*} X; X \in \mathfrak{X}(L) \}, \qquad
H_{\iota} = \{ \phi \iota_{*} Y + f \xi \circ \iota; Y \in \mathfrak{X}(L), f \in C^{\infty}(L) \}.
\end{align*}
We easily see that 
$V_{\iota} = \ker ((d \pi)_{\iota} : T_{\iota} \mathcal{P} \rightarrow T_{\pi (\iota)} \mathcal{A})$
and 
we have a decomposition
$T_{\iota} \mathcal{P} = V_{\iota} \oplus H_{\iota}$. 
As in the proof of \cite[Lemma 3.1]{LotayPacini}, we see that 
the distribution $\iota \mapsto H_{\iota}$ on $\mathcal{P}$ 
is ${\rm Diff}^{\infty}(L)$-invariant. 
Thus 
the distribution $\iota \mapsto H_{\iota}$ defines a connection on 
the principal ${\rm Diff}^{\infty}(L)$-bundle $\mathcal{P}$.

It is known that 
the associated vector bundle 
$
\mathcal{P} \times_{{\rm Diff}^{\infty}(L)} (\mathfrak{X}(L) \times C^{\infty}(L))
$
to the standard action of ${\rm Diff}^{\infty}(L)$ on 
$\mathfrak{X}(L) \times C^{\infty}(L)$
is isomorphic to 
the tangent bundle $T \mathcal{A}$:
\begin{align*}
\mathcal{P} \times_{{\rm Diff}^{\infty}(L)} (\mathfrak{X}(L) \times C^{\infty}(L)) 
\cong T \mathcal{A} 
\end{align*}
via $[\iota, (Y, f)] \mapsto (d \pi)_{\iota} (\phi \iota_{*} Y + f \xi \circ \iota)$.
Then the connection on $\mathcal{P}$ induces a connection on $T \mathcal{A}$.
We define the geodesic 
$\{ L_{t} \}$ on $\mathcal{A}$ by requiring that 
$\frac{d L_{t}}{dt}$ is parallel with respect to this connection. 

\begin{lem} \label{def of geodesic}
A curve $\{ L_{t} \} \subset \mathcal{A}$ is a 
{\bf geodesic} if and only if there exists a curve 
of affine Legendrian embeddings $\{ \iota_{t} \}$, a fixed vector field $Y \in \mathfrak{X}(L)$
and a function $f \in C^{\infty}(L)$
such that 
$\pi (\iota_{t}) = L_{t}$ and
\begin{align*}
\frac{d \iota_{t}}{dt} = \phi (\iota_{t})_{*} Y + f \xi \circ \iota_{t}.
\end{align*}
This implies that 
$[ (\iota_{t})_{*} Y, \phi  (\iota_{t})_{*} Y + f \xi \circ \iota_{t}] = 0$ 
for all t for which $\{ L_{t} \}$ is defined.
\end{lem}

\begin{proof}
Let $\{ L_{t} \}_{t \in (a, b)} \subset \mathcal{A}$ be a geodesic 
and $\{ x(s) \}_{s \in (c, d)}$ be an integral curve
of $Y$ on $L$. 
Then
\begin{align*}
f: (c, d) \times (a, b) \ni (s, t) \mapsto \iota_{t} (x(s)) \in M
\end{align*}
is an embedded surface in $M$ and 
\begin{align*}
\frac{\partial f}{\partial t} = \phi (\iota_{t})_{*} Y + f \xi \circ \iota_{t}, \qquad
\frac{\partial f}{\partial s} = (\iota_{t})_{*} Y, 
\end{align*}
which imply that they commute.
\end{proof}

\begin{definition} \label{def of convex}
A functional $F : \mathcal{A} \rightarrow \mathbb{R}$ is {\bf convex} 
if and only if 
$\{ F \circ L_{t} \}$
is a convex function in one variable for any geodesic $\{ L_{t} \}$ in $\mathcal{A}$. 
\end{definition}

\begin{rem}
The existence theory of a geodesic for any 
$Y \in \mathfrak{X}(L)$ and $f \in C^{\infty}(L)$ is not known 
as in the case of the standard Riemannian geometry. 
\end{rem}


\section{First variation of the $\phi$-volume} \label{first var}

Let $(M^{2n+1}, g, \eta, \xi, \phi)$ be a $(2n+1)$-dimensional Sasakian manifold. 
Let $\iota : L^{n} \hookrightarrow M$ 
be an affine Legendrian immersion 
of an oriented $n$-dimensional manifold $L$. 
For simplicity, we identify 
$\iota_{*} X$ with $X$ for $X \in TL$. 

Fix a point $x \in L$.
Let 
$\{e_{1}, \cdots, e_{n} \}$ 
be an orthonormal basis of $T_{x}L$. 
Since $\iota$ is affine Legendrian, 
$\{ e_{1}, \cdots, e_{n}, \phi (e_{1}), \cdots, \phi (e_{n}), \xi \}$
is the basis of $T_{\iota (x)} M$. 
Let   
$\{e^{1}, \cdots, e^{n}, f^{1}, \cdots, f^{n}, \eta^{*} \} \subset T^{*}_{\iota (x)} M$
be the dual basis.
We easily see the following. 

\begin{lem} \label{basic relation of dual coframe}
Use the notation in (\ref{canonical proj}). We have 
\begin{align*}
e^{i} &= g(\pi_{L} (\cdot), e_{i}), \qquad
\eta^{*} = g(\xi, \pi_{\xi} (\cdot)) = \eta (\pi_{\xi} (\cdot)) = \eta - \eta \circ \pi_{L},\\
e^{i} &= f^{i} \circ \phi, \qquad
f^{i} = -e^{i} \circ \phi. 
\end{align*}
In particular, we have 
\begin{align*}
\eta \circ \pi_{L} \circ \phi = - \eta^{*} \circ \phi.
\end{align*}
\end{lem}

Now, we compute the first variation of the $\phi$-volume form.
We first give all the statements in this section and then prove them. 

\begin{prop} \label{1st var phi vol form}
Let 
$\iota_{t} : L \hookrightarrow M$ 
be a one-parameter family of affine Legendrian immersions 
satisfying $\iota_{0} = \iota$. 
Set $\frac{\partial \iota_{t}}{\partial t}|_{t=0} = Z \in C^{\infty}(L, \iota^{*} TM)$. 
Then 
at $x \in L$
we have 
\begin{align*}
\left. \frac{\partial}{\partial t} {\rm vol}_{\phi} [\iota_{t}] \right|_{t=0} 
= 
\left(
\sum_{i=1}^{n}
e^{i} (\nabla_{e_{i}} Z) 
- \eta^{*} (\phi (Z))
\right) 
{\rm vol}_{\phi} [\iota]. 
\end{align*}
\end{prop}

\begin{definition} \label{def of H phi}
Define the vector field $H_{\phi} \in \mathfrak{X}(L)$ on $L$ by 
\begin{align*}
H_{\phi} =  
-\left( 
\phi {\rm tr}_{L} (\pi^{t}_{\phi} \nabla \pi^{t}_{L})
\right)^{\top}
+ 
\xi^{\top}, 
\end{align*}
where 
${\rm tr}_{L}$ is a trace on $L$, 
$\top: \iota^{*} TM  \rightarrow \iota_{*} TL$ is the tangential projection defined by 
the orthogonal decomposition of $\iota^{*} TM$ by the metric $g$ 
and 
\begin{align*}
\pi^{t}_{L} : \iota^{*} TM  \rightarrow (\phi (\iota_{*} TL) \oplus \mathbb{R} \xi \circ \iota )^{\perp},
\end{align*}
where 
$(\phi (\iota_{*} TL) \oplus \mathbb{R} \xi \circ \iota )^{\perp}$ 
is the orthogonal complement of $\phi (\iota_{*} TL) \oplus \mathbb{R} \xi \circ \iota$
with respect to $g$,
is the transposed operator of $\pi_{L}$ defined in (\ref{canonical proj})
via the metric $g$, namely
\begin{align*}
g(\pi^{t}_{L} X, Y) = g(X, \pi_{L} Y)
\end{align*}
for any $X, Y \in \iota^{*} TM$. 
Similarly, we can define transposed operators 
$\pi^{t}_{\phi}: \iota^{*} TM  \rightarrow (\iota_{*} TL \oplus \mathbb{R} \xi \circ \iota)^{\perp}$
and $\pi^{t}_{\xi} : \iota^{*} TM  \rightarrow (\iota_{*} TL \oplus \phi (\iota_{*} TL))^{\perp}$
of $\pi_{\phi}$ and $\pi_{\xi}$, respectively. 
\end{definition}

The vector field $\phi H_{\phi}$ is a generalization of a mean curvature vector. 
See Remark \ref{MC in Leg case}.

\begin{cor} \label{1st var cor}
Let $X, Y \in \mathfrak{X}(L)$ be vector fields on $L$ 
and $f \in C^{\infty}(L)$ be a smooth function. 
Then we have 
\begin{align*}
\sum_{i=1}^{n} e^{i} (\nabla_{e_{i}} (X+\phi Y + f \xi))  
=
\frac{{\rm div} (\rho_{\phi} [\iota] X)}{\rho_{\phi}[\iota]} 
- 
g (Y, H_{\phi}) + \eta (Y). 
\end{align*}
\end{cor}

From 
Proposition \ref{1st var phi vol form} and Corollary \ref{1st var cor}, 
we immediately see the following first variation formula 
of the $\phi$-volume. 

\begin{prop} \label{1st var phi Vol}

Use the notation of Proposition \ref{1st var phi vol form}. 
Suppose that $L$ is compact and 
$\frac{\partial \iota_{t}}{\partial t}|_{t=0} = Z
= \phi Y + f \xi$, 
where $Y \in \mathfrak{X}(L)$ is a vector field on $L$ 
and $f \in C^{\infty}(L)$ is a smooth function.
Then we have 
\begin{align*}
\left. \frac{d}{dt} {\rm Vol}_{\phi} [\iota_{t}] \right|_{t=0} 
= 
- \int_{L} g (Y, H_{\phi}) {\rm vol}_{\phi} [\iota].
\end{align*}
In particular, 
$\iota$ is a  critical point for the $\phi$-volume  
if and only if $H_{\phi} = 0$. 
\end{prop}
Note that $f$ does not appear in this formula. 
We call an immersion {\bf $\phi$-minimal} 
if $H_{\phi} = 0$.

\begin{rem} \label{MC in Leg case}
Suppose that $L$ is Legendrian. 
Let $H$ be the mean curvature vector of $L$.
Then we have 
\begin{align*}
H_{\phi} = - \phi H, \qquad \phi H_{\phi} = H.
\end{align*}
\end{rem}

\begin{proof}[Proof of Proposition \ref{1st var phi vol form}]

Denote by ${\rm vol}_{\phi} [\iota_{t}]$ the $\phi$-volume of $\iota_{t}$. 
Let $\{e_{1}(t), \cdots, e_{n}(t) \}$ be an oriented orthonormal basis of $T_{x} L$ 
with respect to $\iota_{t}^{*} g$
and 
$\{e^{1}(t), \cdots, e^{n}(t) \} \subset T^{*}_{x} L$ be the dual basis. 
Then 
we have
\begin{align*}
{\rm vol}_{\phi} [\iota_{t}]_{x}
&= \rho_{\phi}[\iota_{t}] (x) e^{1}(t) \wedge \cdots \wedge e^{n}(t), \\
\rho_{\phi} [\iota_{t}] (x)
&=
\sqrt{
({\rm vol}_{g})_{\iota_{t}(x)}
(
(\iota_{t})_{*} e_{1}(t), \cdots, (\iota_{t})_{*} e_{n}(t), 
-\xi_{\iota_{t}(x)}, 
\phi (\iota_{t})_{*} e_{1}(t), \cdots, \phi (\iota_{t})_{*} e_{n}(t)
)
}.
\end{align*}
Since 
\begin{align*}
(\iota_{t})_{*} (e_{1} \wedge \cdots \wedge e_{n})
=
(e^{1}(t) \wedge \cdots \wedge e^{n}(t))(e_{1}, \cdots, e_{n})) \cdot 
(\iota_{t})_{*} (e_{1}(t) \wedge \cdots \wedge e_{n}(t)), 
\end{align*}
it follows that 
\begin{align*}
{\rm vol}_{\phi} [\iota_{t}]_{x} &= \rho_{\phi} (t) (x) \cdot {\rm vol}_{\iota^{*}g}, 
\end{align*}
where 
\begin{align*}
\rho_{\phi} (t) (x)
=
\sqrt{
({\rm vol}_{g})_{\iota_{t}(x)}
(
(\iota_{t})_{*} e_{1}, \cdots, (\iota_{t})_{*} e_{n}, 
-\xi_{\iota_{t}(x)}, 
\phi (\iota_{t})_{*} e_{1}, \cdots, \phi (\iota_{t})_{*} e_{n}
)
}.
\end{align*}
Thus we may consider $\left. \frac{\partial}{\partial t} \rho_{\phi}(t) (x) \right|_{t=0}$. 
Set 
\begin{align*}
\nabla_{Z} e_{i} &= \nabla_{\frac{\partial}{\partial t}} (\iota_{t})_{*} e_{i} |_{t=0}, \qquad
\nabla_{Z} (\phi e_{i}) = \nabla_{\frac{\partial}{\partial t}} (\phi (\iota_{t})_{*} e_{i}) |_{t=0}. 
\end{align*}
Since the volume form is parallel, we have 
\begin{align*}
& \left. \frac{\partial}{\partial t} \rho_{\phi}(t) (x) \right|_{t=0}
\\
=&
\frac{\sum_{i=1}^{n}
({\rm vol}_{g})_{\iota (x)}
(
(e_{1}, \cdots, \nabla_{Z} e_{i}, \cdots, e_{n}, -\xi_{\iota (x)}, 
\phi e_{1}, \cdots, \phi e_{n}
)
}{2 \rho_{\phi}[\iota]}\\
&+
\frac{\sum_{i=1}^{n}
({\rm vol}_{g})_{\iota (x)}
( e_{1}, \cdots, e_{n}, 
-\xi_{\iota (x)}, 
\phi e_{1}, \cdots, 
\nabla_{Z} (\phi e_{i}), \cdots,
\phi e_{n} 
)
}{2 \rho_{\phi}[\iota]}\\
&+
\frac{
({\rm vol}_{g})_{\iota (x)}
( e_{1}, \cdots, e_{n}, 
- \nabla_{\frac{\partial}{\partial t}} \xi_{\iota_{t}(x)} |_{t=0}, 
\phi e_{1}, \cdots, \phi e_{n}
)
}{2 \rho_{\phi}[\iota]}.
\end{align*}
Using the notation at the beginning of Section \ref{first var}, 
we have 
\begin{align*}
&
\left.
\frac{\partial}{\partial t} \rho_{\phi}(t) (x) \right|_{t=0}
\\
=&
\frac{\sum_{i=1}^{n}
e^{i} (\nabla_{Z} e_{i}) \rho_{\phi}[\iota]^{2} 
}{2 \rho_{\phi}[\iota]}
+
\frac{\sum_{i=1}^{n}
f^{i} (\nabla_{Z} (\phi e_{i})) \rho_{\phi}[\iota]^{2} 
}{2 \rho_{\phi}[\iota]}
-
\frac{
\eta^{*} (\phi (Z)) \rho_{\phi}[\iota]^{2} 
}{2 \rho_{\phi}[\iota]}. 
\end{align*}
By Lemmas \ref{Sasaki eq2} and \ref{basic relation of dual coframe}, 
it follows that
\begin{align*}
\nabla_{Z} (\phi e_{i})
&= g(Z, e_{i}) \xi - \eta (e_{i}) Z + \phi (\nabla_{Z} e_{i}), \\
\sum_{i=1}^{n} f^{i} (\nabla_{Z} (\phi e_{i}))
&=
\sum_{i=1}^{n} e^{i} (\nabla_{Z} e_{i})
- \eta^{*} (\phi (Z)).
\end{align*}
Thus we obtain
\begin{align*}
\left.
\frac{\partial}{\partial t} \rho_{\phi}(t) (x) \right|_{t=0}
=
\left(
\sum_{i=1}^{n}
e^{i} (\nabla_{Z} e_{i}) 
- \eta^{*} (\phi (Z))
\right) 
\rho_{\phi}[\iota].
\end{align*}

Here, 
let 
$(x_{1}, \cdots, x_{n})$ be a 
normal coordinate at $x \in L$ of $L$ 
satisfying  
$e_{i} = \frac{\partial}{\partial x_{i}}$
at $x \in L$. 
Define the map 
$\tilde{\iota}: L \times (-\epsilon, \epsilon) 
\rightarrow M$
by
$\tilde{\iota} (x, t) = \iota_{t} (x)$. 
Then 
$\tilde{\iota}_{*} (\frac{\partial}{\partial t})|_{t=0} = Z$. 
Since we may regard $(x_{1}, \cdots, x_{n}, t)$
as the local coordinate of $L \times (-\epsilon, \epsilon)$ near 
$(x, 0)$, 
we have $\nabla_{Z} e_{i} = \nabla_{e_{i}} Z$. 
Thus we obtain the statement. 
\end{proof}

\begin{proof}[Proof of Corollary \ref{1st var cor}]

Setting $Z=X \in \mathfrak{X}(L)$ 
in Proposition \ref{1st var phi vol form}, we have 
\begin{align*}
\left. \frac{\partial}{\partial t} {\rm vol}_{\phi} [\iota_{t}] \right|_{t=0} 
= 
\sum_{i=1}^{n} e^{i} (\nabla_{e_{i}} X) {\rm vol}_{\phi} [\iota]. 
\end{align*}
By Lemma \ref{diff equiv aff Leg}, it follows that 
\begin{align*}
\left. \frac{\partial}{\partial t} {\rm vol}_{\phi} [\iota_{t}] \right|_{t=0} 
= 
L_{X} {\rm vol}_{\phi} [\iota] 
= 
{\rm div}(\rho_{\phi}[\iota] X) {\rm vol}_{\iota^{*}g}.
\end{align*}
Thus we obtain 
\begin{align*}
\sum_{i=1}^{n} e^{i} (\nabla_{e_{i}} X) 
=
\frac{{\rm div} (\rho_{\phi} [\iota] X)}{\rho_{\phi}[\iota]}. 
\end{align*}
Using the notation of (\ref{canonical proj}), we have
\begin{align*}
\sum_{i=1}^{n} e^{i} (\nabla_{e_{i}} (\phi Y))
&=
\sum_{i=1}^{n} g(\pi_{L} (\nabla_{e_{i}} (\phi Y)), e_{i}) \\
&=
-\sum_{i=1}^{n} g( \pi_{\phi} (\phi Y), \nabla_{e_{i}} (\pi^{t}_{L} e_{i})) 
=
g (Y, -H_{\phi} + \xi^{\top}).
\end{align*}
by Lemma \ref{basic relation of dual coframe}. 
It is easy to show that 
$
\sum_{i=1}^{n} e^{i} (\nabla_{e_{i}} (f \xi))  
=0
$
and the proof is done. 
\end{proof}

Note that we can also prove 
Corollary \ref{1st var cor} 
by a direct computation.

\begin{proof}[Proof of Remark \ref{MC in Leg case}]
Since $L$ is Legendrian, 
we have $\pi_{L}^{t} = \pi_{L}$ and $\pi_{\phi}^{t} = \pi_{\phi}$. 
Then 
\begin{align*}
H_{\phi} = - \left( \phi \sum_{i=1}^{n} \pi_{\phi} \nabla_{e_{i}} e_{i} \right)^{\top}
= - \phi \sum_{i=1}^{n} \pi_{\phi} \nabla_{e_{i}} e_{i}.
\end{align*}

Let 
$\pi_{\perp}: \iota^{*}TM = \iota_{*} TL \oplus (\iota_{*} TL)^{\perp} \rightarrow (\iota_{*} TL)^{\perp}$
be the normal projection with respect to $g$. 
Then we see that 
\begin{align*}
H &= \pi_{\perp} \left( \sum_{i=1}^{n} \nabla_{e_{i}} e_{i} \right) 
= \pi_{\phi} \left(\sum_{i=1}^{n} \nabla_{e_{i}} e_{i} \right) 
+ \pi_{\xi} \left(\sum_{i=1}^{n} \nabla_{e_{i}} e_{i} \right), \\
\pi_{\xi} \left(\nabla_{e_{i}} e_{i} \right) 
&=
g(\nabla_{e_{i}} e_{i}, \xi) \xi 
=
e_{i} (g(e_{i}, \xi)) \xi + g(e_{i}, \phi e_{i}) \xi 
= 0,
\end{align*}
which implies the statement. 
\end{proof}

\section{Second variation of the $\phi$-volume}

In this section, 
we compute the second variation of the $\phi$-volume 
and prove Theorems \ref{2nd var phi Vol},  
\ref{thm stablity}, \ref{thm convexity} and \ref{thm obstruction}. 
Use the notation in Section \ref{first var}. 
First, we compute the second variation of the $\phi$-volume form.

\begin{prop} \label{2nd var phi vol form}
Let 
$\iota_{t} : L \hookrightarrow M$ 
be a one-parameter family of affine Legendrian immersions 
satisfying $\iota_{0} = \iota$. 
Set $\frac{\partial \iota_{t}}{\partial t}|_{t=0} = Z \in C^{\infty}(L, \iota^{*}TM).$ 
Then at $x \in L$ we have 
\begin{align*}
\left. \frac{\partial^{2}}{\partial t^{2}} {\rm vol}_{\phi} [\iota_{t}] \right|_{t=0} 
=& 
\left \{ 
-2 \eta^{*} (\phi (Z)) \sum_{i} e^{i} (\nabla_{e_{i}}Z) \right. \\
&- \sum_{i, j} e^{i} (\nabla_{e_{j}} Z) e^{j} (\nabla_{e_{i}} Z)
+ \sum_{i, j} f^{i} (\nabla_{e_{j}} Z) f^{j} (\nabla_{e_{i}} Z)  \\
&+
\sum_{i} e^{i} (R(Z, e_{i})Z + \nabla_{e_{i}} \nabla_{Z} Z) \\
&- \eta(\pi_{L} Z) \eta^{*}(Z) - \eta^{*}(\phi (\nabla_{Z} Z)) - g(Z, \pi_{\phi} Z)\\
&-2 \sum_{i} f^{i}(Z) \eta^{*}(\nabla_{e_{i}} Z)
+2 \sum_{i} e^{i}(Z) \eta^{*}(\phi (\nabla_{e_{i}} Z)) \\
&\left.  + \left( \sum_{i} e^{i} (\nabla_{e_{i}} Z) \right)^{2}
\right \}
{\rm vol}_{\phi} [\iota].
\end{align*}
where
$R$ is the curvature tensor of $(M, g)$.

In particular, 
when $Z= X \in \mathfrak{X}(L)$, we have 
\begin{align*}
\frac{\partial^{2}}{\partial t^{2}} {\rm vol}_{\phi} [\iota_{t}]|_{t=0} 
=& {\rm div} ({\rm div} (\rho_{\phi} X) X) {\rm vol}_{\iota^{*} g}\\
=& 
\left \{ 
- \sum_{i, j} e^{i} (\nabla_{e_{j}} X) e^{j} (\nabla_{e_{i}} X)
+ \sum_{i, j} f^{i} (\nabla_{e_{j}} X) f^{j} (\nabla_{e_{i}} X)
\right. 
\\
&+
\sum_{i} e^{i} (R(X, e_{i})X + \nabla_{e_{i}} \nabla_{X} X) \\
&+ \eta^{*}(\phi (\nabla_{X} X))
\left.  + \left( \sum_{i} e^{i} (\nabla_{e_{i}} X) \right)^{2}
\right \}
{\rm vol}_{\phi} [\iota].
\end{align*}
\end{prop}

By Proposition \ref{2nd var phi vol form}, 
we obtain the second variation formula of the $\phi$-volume (Theorem \ref{2nd var phi Vol}).

\begin{proof}[Proof of Proposition \ref{2nd var phi vol form}]

Use the notation in the proof of Proposition \ref{1st var phi vol form}. 
Set 
\begin{align*}
\nabla_{Z} e_{i} &= \nabla_{\frac{\partial}{\partial t}} (\iota_{t})_{*} e_{i} |_{t=0}, \qquad
\nabla_{Z} (\phi e_{i}) = \nabla_{\frac{\partial}{\partial t}} (\phi (\iota_{t})_{*} e_{i}) |_{t=0}, \\
\nabla_{Z} \nabla_{Z} e_{i} &=
\nabla_{\frac{\partial}{\partial t}} \nabla_{\frac{\partial}{\partial t}} 
(\iota_{t})_{*} e_{i} |_{t=0}, \qquad
\nabla_{Z} \nabla_{Z} (\phi e_{i}) =
\nabla_{\frac{\partial}{\partial t}} \nabla_{\frac{\partial}{\partial t}} 
(\phi (\iota_{t})_{*} e_{i}) |_{t=0}.
\end{align*}
Then we have 
\begin{align*}
\left.
\frac{\partial^{2}}{\partial t^{2}} \rho_{\phi}^{2} (t) 
\right|_{t=0}
= \sum_{i=1}^{11} h_{i}, 
\end{align*}
where 
\begin{align*}
h_{1} &= 
\sum_{i \neq j} {\rm vol}_{g} (e_{1}, \cdots, \nabla_{Z} e_{i}, \cdots, \nabla_{Z} e_{j}, \cdots, e_{n},
-\xi, \phi e_{1}, \cdots, \phi e_{n}), \\
h_{2} &= 
\sum_{i=1}^{n} {\rm vol}_{g} (e_{1}, \cdots, \nabla_{Z} \nabla_{Z} e_{i}, \cdots, e_{n},
-\xi, \phi e_{1}, \cdots, \phi e_{n}), \\
h_{3} &= 
\sum_{i=1}^{n} {\rm vol}_{g} (e_{1}, \cdots, \nabla_{Z} e_{i}, \cdots, e_{n},
-\nabla_{Z} \xi, \phi e_{1}, \cdots, \phi e_{n}), \\
h_{4} &= 
\sum_{i,j=1}^{n} {\rm vol}_{g} (e_{1}, \cdots, \nabla_{Z} e_{i}, \cdots, e_{n},
-\xi, \phi e_{1}, \cdots, \nabla_{Z} (\phi e_{j}), \cdots, \phi e_{n}), \\
h_{5} &=
\sum_{i \neq j} {\rm vol}_{g} (e_{1}, \cdots, e_{n},
-\xi, \phi e_{1}, \cdots, \nabla_{Z} (\phi e_{i}), \cdots, \nabla_{Z} (\phi e_{j}), \cdots, \phi e_{n}), \\
h_{6} &= 
\sum_{i=1}^{n} {\rm vol}_{g} (e_{1}, \cdots, e_{n},
-\xi, \phi e_{1}, \cdots, \nabla_{Z} \nabla_{Z} (\phi e_{i}), \cdots, \phi e_{n}), \\
h_{7} &=
\sum_{i=1}^{n} {\rm vol}_{g} (e_{1}, \cdots, e_{n},
-\nabla_{Z} \xi, \phi e_{1}, \cdots, \nabla_{Z} (\phi e_{i}), \cdots, \phi e_{n}), \\
h_{8} &= 
\sum_{i,j=1}^{n} {\rm vol}_{g} (e_{1}, \cdots, \nabla_{Z} e_{i}, \cdots, e_{n},
-\xi, \phi e_{1}, \cdots, \nabla_{Z} (\phi e_{j}), \cdots, \phi e_{n}), \\
h_{9} &= 
\sum_{i=1}^{n} {\rm vol}_{g} (e_{1}, \cdots, \nabla_{Z} e_{i}, \cdots, e_{n},
-\nabla_{Z} \xi, \phi e_{1}, \cdots, \phi e_{n}), \\
h_{10} &= 
\sum_{i=1}^{n} {\rm vol}_{g} (e_{1}, \cdots, e_{n},
-\nabla_{Z} \xi, \phi e_{1}, \cdots, \nabla_{Z} (\phi e_{i}), \cdots, \phi e_{n}), \\
h_{11} &=
{\rm vol}_{g} (e_{1}, \cdots, e_{n},
-\nabla_{Z} \nabla_{Z} \xi, \phi e_{1}, \cdots, \phi e_{n}).
\end{align*}
We divide $h_{i}$'s into the following four classes and 
compute in each class.

\begin{itemize}
\item[class 1:] $h_{1}, h_{5}$,
\item[class 2:] $h_{2}, h_{6}, h_{11},$
\item[class 3:] $h_{3}= h_{9}, h_{7} = h_{10},$
\item[class 4:] $h_{4} = h_{8}$.
\end{itemize}

We simplify $h_{i}$'s by Lemmas \ref{Sasaki eq1}, \ref{Sasaki eq2} and \ref{basic relation of dual coframe}. 
First, we compute $h_{i}$'s in class 1.
It is easy to see that 
\begin{align*}
h_{1} = \left \{ 
\left( \sum_{i} e^{i} (\nabla_{Z} e_{i}) \right)^{2} 
- 
\sum_{i, j} e^{i} (\nabla_{Z} e_{j}) e^{j} (\nabla_{Z} e_{i})
\right \}
\rho_{\phi}[\iota]^{2}.
\end{align*}
Since 
\begin{align*}
\nabla_{Z} (\phi e_{i}) &= g(Z, e_{i}) \xi - \eta (e_{i}) Z + \phi (\nabla_{Z} e_{i}), \\
f^{i}(\nabla_{Z} (\phi e_{j})) &= - \eta (e_{j}) f^{i}(Z) + e^{i} (\nabla_{Z} e_{j}),
\end{align*}
we have 
\begin{align*}
h_{5} &= 
\sum_{i \neq j} 
\left \{
f^{i}(\nabla_{Z} (\phi e_{i})) f^{j}(\nabla_{Z} (\phi e_{j}))
- 
f^{j}(\nabla_{Z} (\phi e_{i})) f^{i}(\nabla_{Z} (\phi e_{j}))
\right \}
\rho_{\phi}[\iota]^{2}\\
&=
\sum_{i, j} 
\left \{
-2 \eta (e_{i}) f^{i} (Z) e^{j}(\nabla_{Z} e_{j}) 
+2 \eta (e_{i}) f^{j}(Z) e^{i}(\nabla_{Z} e_{j})
\right \}
\rho_{\phi}[\iota]^{2}
+ h_{1}\\
&=
\left \{
2 \eta (\pi_{L} \phi (Z)) 
\sum_{j} e^{j}(\nabla_{Z} e_{j}) 
+2 \sum_{j} \eta (\pi_{L} \nabla_{Z} e_{j}) f^{j}(Z) 
\right \}
\rho_{\phi}[\iota]^{2}
+ h_{1}\\
&=
\left \{
- 2 \eta^{*} (\phi (Z)) 
\sum_{j} e^{j}(\nabla_{Z} e_{j}) 
+2 \sum_{j} \eta (\pi_{L} \nabla_{Z} e_{j}) f^{j}(Z) 
\right \}
\rho_{\phi}[\iota]^{2}
+ h_{1}.
\end{align*}

Next, we compute $h_{i}$'s in class 2. 
We easily see that 
\begin{align*}
h_{2} &= \sum_{i} e^{i} (\nabla_{Z} \nabla_{Z} e_{i}) \rho_{\phi}[\iota]^{2}, \qquad
h_{6} = \sum_{i} f^{i} (\nabla_{Z} \nabla_{Z} (\phi e_{i})) \rho_{\phi}[\iota]^{2}, \\
h_{11} &=  
\eta^{*} (\nabla_{Z} \nabla_{Z} \xi) \rho_{\phi}[\iota]^{2}.
\end{align*}
Set $\nabla_{Z} Z = \nabla_{\frac{\partial}{\partial t}} \tilde{\iota}_{*} (\frac{\partial}{\partial t}) |_{t=0}$, 
where $\tilde{\iota}$ is given 
in proof of Proposition \ref{1st var phi Vol}. 
Since 
\begin{align*}
\nabla_{Z} \nabla_{Z} (\phi e_{i})
=&
Z(g(Z, e_{i})) \xi - g(Z, e_{i}) \phi (Z)
+
\{ g(\phi(Z), e_{i}) - \eta (\nabla_{Z} e_{i}) \} Z \\
&- \eta (e_{i}) \nabla_{Z} Z 
+g(Z, \nabla_{Z} e_{i}) \xi - \eta (\nabla_{Z} e_{i}) Z +
\phi (\nabla_{Z} \nabla_{Z} e_{i}), \\
\nabla_{Z} \nabla_{Z} \xi
=&
- g(Z, Z) \xi + \eta (Z) Z - \phi (\nabla_{Z} Z), 
\end{align*}
we obtain 
\begin{align*}
h_{6}
=&
\left \{
- g(Z, \pi_{L} Z) 
- g(Z, \pi_{\phi} Z) \right.
\\
&\left. -2 \sum_{i} \eta (\nabla_{Z} e_{i}) f^{i}(Z)
- \eta^{*} (\phi (\nabla_{Z} Z))
\right \}
\rho_{\phi}[\iota]^{2}
+ 
h_{2}, \\
h_{11} =& 
\{ - g(Z, Z) + \eta (Z) \eta^{*} (Z) - \eta^{*} (\phi (\nabla_{Z} Z)) \} \rho_{\phi}[\iota]^{2}\\
=& 
\{ - g(Z, \pi_{L} Z) - g(Z, \pi_{\phi} Z) - \eta^{*} (\phi (\nabla_{Z} Z)) \} \rho_{\phi}[\iota]^{2}.
\end{align*}

Compute $h_{i}$'s in class 3 to obtain 
\begin{align*}
h_{3} = h_{9}
=&
\left \{
- \sum_{i} e^{i} (\nabla_{Z} e_{i}) \eta^{*} (\phi(Z))
-
\eta^{*} (\nabla_{Z} e_{i}) f^{i} (Z)
\right \} \rho_{\phi}[\iota]^{2}\\
h_{7} = h_{10}
=&
\left \{
- \sum_{i} \eta^{*}(\phi(Z)) f^{i} (\nabla_{Z} (\phi e_{i})) 
+ 
f^{i} (\phi (Z)) \eta^{*} (\nabla_{Z} (\phi e_{i})) 
\right \} \rho_{\phi}[\iota]^{2}\\
=&
\left \{
\sum_{i} \eta^{*}(\phi(Z)) (\eta (e_{i}) f^{i}(Z) - e^{i} (\nabla_{Z} e_{i})) \right. \\
&\left. +
e^{i} (Z) 
\left( g(Z, e_{i}) - \eta (e_{i}) \eta^{*} (Z) + \eta^{*} (\phi (\nabla_{Z} e_{i}) \right) 
\right \} \rho_{\phi}[\iota]^{2}\\
=&
\left \{
\eta^{*}(\phi(Z))^{2} 
-\eta^{*}(\phi(Z)) \sum_{i}  e^{i} (\nabla_{Z} e_{i}) \right. \\
&
\left. 
+g(Z, \pi_{L} Z) - \eta (\pi_{L} Z) \eta^{*}(Z) + \sum_{i} e^{i}(Z) \eta^{*} (\phi (\nabla_{Z} e_{i})) 
\right \} \rho_{\phi}[\iota]^{2}. 
\end{align*}

Compute $h_{i}$'s in class 4 to obtain 
\begin{align*}
h_{4} = h_{8}
=& \sum_{i, j} \left \{ 
e^{i} (\nabla_{Z} e_{i}) f^{j} (\nabla_{Z} (\phi e_{j}))
- 
f^{j} (\nabla_{Z} e_{i}) e^{i} (\nabla_{Z} (\phi e_{j}))
\right \}
\rho_{\phi}[\iota]^{2}\\
=&
\sum_{i, j} \left \{ 
e^{i} (\nabla_{Z} e_{i}) ( - \eta (e_{j}) f^{j}(Z) +e^{j} (\nabla_{Z} e_{j})) \right. \\
&\left. - 
f^{j} (\nabla_{Z} e_{i}) 
\left( - \eta (e_{j}) e^{i}(Z) + e^{i} (\phi (\nabla_{Z} e_{j})) \right)
\right \}
\rho_{\phi}[\iota]^{2}\\
=&
\left \{ 
- \eta^{*} (\phi (Z))
\sum_{i} e^{i} (\nabla_{Z} e_{i}) 
+ 
\left(\sum_{i} e^{i} (\nabla_{Z} e_{i}) \right)^{2} \right.\\
& \left. 
+ 
\sum_{i} \eta^{*} (\phi (\nabla_{Z} e_{i}))e^{i}(Z) + 
\sum_{i} f^{i}(\nabla_{Z} e_{j}) f^{j} (\nabla_{Z} e_{i})
\right \}
\rho_{\phi}[\iota]^{2}.
\end{align*}
Then by a direct computation, we obtain
\begin{align*}
\left.
\frac{\partial^{2}}{\partial t^{2}} \rho_{\phi}(t) \right|_{t=0}
=&
\frac{1}{2 \rho_{\phi}[\iota]} 
\left( 
\left. 
\frac{\partial^{2} (\rho_{\phi}(t)^{2})}{\partial t^{2}}  \right|_{t=0}
- 2
\left(
\frac{\partial \rho_{\phi}(t)}{\partial t}
\right)^{2}
\right)\\
=&
\frac{1}{2 \rho_{\phi}[\iota]} 
\left( 
\sum_{i=1}^{11} h_{i}
- 2
\left(
\sum_{i} e^{i} (\nabla_{Z} e_{i}) - \eta^{*} (\phi (Z))
\right)^{2} \rho_{\phi}[\iota]^{2}
\right) \\
=&
\left \{ 
-2 \eta^{*} (\phi (Z)) \sum_{i} e^{i} (\nabla_{Z} e_{i}) \right. \\
&- \sum_{i, j} e^{i} (\nabla_{Z} e_{j}) e^{j} (\nabla_{Z} e_{i})
+ \sum_{i, j} f^{i} (\nabla_{Z} e_{j}) f^{j} (\nabla_{Z} e_{i})  \\
&+
\sum_{i} e^{i} (\nabla_{Z} \nabla_{Z} e_{i}) \\
&- \eta(\pi_{L} Z) \eta^{*}(Z) - \eta^{*}(\phi (\nabla_{Z} Z)) - g(Z, \pi_{\phi} Z)\\
&-2 \sum_{i} f^{i}(Z) \eta^{*}(\nabla_{Z} e_{i})
+2 \sum_{i} e^{i}(Z) \eta^{*}(\phi (\nabla_{Z} e_{i})) \\
&\left.  + \left( \sum_{i} e^{i} (\nabla_{Z} e_{i}) \right)^{2}
\right \}
{\rm vol}_{\phi} [\iota]. 
\end{align*}
Here, 
take the normal coordinate $(x_{1}, \cdots, x_{n})$
at $x \in L$ in the proof of Proposition \ref{1st var phi vol form}. 
Then we have 
\begin{align*}
\nabla_{Z} e_{i} = \nabla_{e_{i}} Z, \qquad
\nabla_{Z} \nabla_{Z} e_{i} = R(Z, e_{i})Z + \nabla_{e_{i}} \nabla_{Z} Z
\end{align*}
at $x$, which give the first statement of Proposition \ref{2nd var phi vol form}.

When $Z=X \in \mathfrak{X}(L)$, 
Lemma \ref{diff equiv aff Leg} implies that 
\begin{align*}
\left.
\frac{\partial^{2}}{\partial t^{2}} \rho_{\phi}[\iota] \right|_{t=0}
&=
L_{X} L_{X} {\rm vol}_{\phi}[\iota]\\
&=
d (i(X) d (i(X) \rho_{\phi}[\iota] {\rm vol}_{\iota^{*}g}))
=
{\rm div} ({\rm div}(\rho_{\phi}[\iota] X)X) {\rm vol}_{\iota^{*}g},
\end{align*}
which gives the second statement. 
\end{proof}

We compute the next lemma to
prove Theorem \ref{2nd var phi Vol}.

\begin{lem} \label{computation 2nd var 1}
Suppose that 
$Z= \phi Y + f\xi$, 
where $Y \in \mathfrak{X}(L)$ and $f \in C^{\infty}(L)$. 
Then we have 
\begin{align*}
\sum_{i, j} e^{i} (\nabla_{e_{j}} Z) e^{j} (\nabla_{e_{i}} Z)
&=
n \eta(Y)^{2} + 
2 \eta(Y) \sum_{i} f^{i} (\nabla_{e_{i}} Y) 
+
\sum_{i, j} f^{i} (\nabla_{e_{j}} Y) f^{j} (\nabla_{e_{i}} Y),\\
\sum_{i, j} f^{i} (\nabla_{e_{j}} Z) f^{j} (\nabla_{e_{i}} Z)
&=
n f^{2} 
- 2 f \sum_{i} e^{i} (\nabla_{e_{i}} Y) 
+
\sum_{i, j} e^{i} (\nabla_{e_{j}} Y) e^{j} (\nabla_{e_{i}} Y),
\end{align*}
\begin{align*}
\sum_{i} e^{i}(R(Z, e_{i})Z + R(Y, e_{i})Y) 
=
- {\rm Ric} (Y, Y) + g(Y,Y)+ n \eta (Y)^{2} - n f^{2}, 
\end{align*}
where 
${\rm Ric}$ is the Ricci curvature of $(M,g)$.
\end{lem}

\begin{proof}[Proof of Lemma \ref{computation 2nd var 1}]
The first two equations follow by the next equation:
\begin{align*}
\nabla_{X} Z = g(X, Y) \xi - \eta(Y) X 
+ \phi (\nabla_{X} Y) + X(f) \xi -f \phi (X), 
\end{align*}
where $X \in \mathfrak{X}(L)$ is a vector field on $L$.

We prove the third equation. 
By Lemma  \ref{Sasaki eq2},
we see that 
\begin{align*}
\sum_{i} e^{i} (R(Z, e_{i})Z)
=&
\sum_{i} e^{i} (R(Z, e_{i})(\phi Y + f \xi))\\
=& 
\sum_{i} 
e^{i}
\left( \phi (R(Z, e_{i})Y) - g(e_{i}, Y) \phi (Z)
+g(\phi Z, Y) e_{i} \right. \\
&\left. + g(Z, Y) \phi (e_{i}) - g(\phi e_{i}, Y) Z 
+ f (\eta (e_{i}) Z - \eta (Z) e_{i})
\right)\\
=&
\sum_{i} e^{i} \circ \phi (R(Z, e_{i})Y) 
+ (-n+1) g(Y,Y) 
+n \eta(Y)^{2} - n f^{2}. 
\end{align*}
The first term is computed as 
\begin{align*}
\sum_{i} e^{i} \circ \phi (R(Z, e_{i})Y) 
=
\sum_{i} g(\pi_{L} \phi R(Z, e_{i}) Y, e_{i})
=
- \sum_{i} g(R(Y, \phi \pi_{L}^{t} e_{i}) Z, e_{i}), 
\end{align*}

\begin{align*}
R(Y, \phi \pi_{L}^{t} e_{i}) Z
=&
\phi (R(Y, \phi \pi_{L}^{t} e_{i})Y) - g(\phi \pi_{L}^{t} e_{i}, Y) \phi (Y)
+ g(\phi Y, Y) \phi \pi_{L}^{t} e_{i} \\
&+ g(Y, Y) \phi^{2} \pi_{L}^{t} e_{i}
- g(\phi^{2} \pi_{L}^{t} e_{i}, Y) Y 
+ f(\eta(\phi \pi_{L}^{t} e_{i}) Y - \eta(Y) \phi \pi_{L}^{t} e_{i})\\
=&
\phi (R(Y, \phi \pi_{L}^{t} e_{i})Y) 
- g(Y, Y) \pi_{L}^{t} e_{i}
+ g(e_{i}, Y) Y 
- f \eta(Y) \phi \pi_{L}^{t} e_{i}.
\end{align*}
Thus we obtain 
\begin{align*}
\sum_{i} e^{i} \circ \phi (R(Z, e_{i})Y) 
=&
\sum_{i}
g(-\phi (R(Y, \phi \pi_{L}^{t} e_{i})Y) 
+ g(Y, Y) \pi_{L}^{t} e_{i}
- g(e_{i}, Y) Y, e_{i})\\
=&
\sum_{i} 
g(R(Y, \phi e_{i}) Y, \phi \pi_{L}^{t} e_{i}) 
+ (n-1) g(Y, Y)\\
=&
- \sum_{i} g(\pi_{L} \phi R(Y, \phi e_{i})Y, e_{i}) 
+ (n-1) g(Y, Y)\\
=&
\sum_{i} f^{i} (R(Y, \phi e_{i})Y) + (n-1) g(Y, Y). 
\end{align*}
Since 
\begin{align*}
\eta^{*} (R(\xi, Y)Y)
&=
g(\pi_{\xi} R(\xi, Y)Y, \xi)
=
g(R(Y, \pi_{\xi}^{t} \xi)\xi, Y)\\
&=
g(\eta(\pi_{\xi}^{t} \xi)Y - \eta (Y) \pi_{\xi}^{t} \xi, Y)
=
g(Y, Y),
\end{align*}
it follows that  
\begin{align*}
&\sum_{i} e^{i} (R(Z, e_{i})Z + R(Y, e_{i})Y) \\
=&
- \sum_{i} (f^{i} (R(\phi e_{i}, Y)Y) + e^{i} (R(e_{i},Y)Y))
+ (n-1) g(Y,Y) \\
&+ (-n+1) g(Y,Y) 
+ n \eta (Y)^{2} - n f^{2}\\
=&
- {\rm Ric} (Y, Y) + g(Y,Y) + n \eta (Y)^{2} - n f^{2}. 
\end{align*}
\end{proof}

\begin{proof}[Proof of Theorem \ref{2nd var phi Vol}]

Set 
$Z= \phi Y + f\xi$.
By Proposition \ref{2nd var phi vol form}
and 
Lemma
\ref{computation 2nd var 1}, 
we have 
\begin{align*}
\left. \frac{\partial^{2}}{\partial t^{2}} {\rm vol}_{\phi} [\iota_{t}] \right|_{t=0} 
=& 
\rho_{\phi}[\iota]
\left \{ 
-2 \eta(Y) \sum_{i} e^{i} (\nabla_{e_{i}}Z) \right. \\
&- \sum_{i, j} e^{i} (\nabla_{e_{j}} Z) e^{j} (\nabla_{e_{i}} Z)
+ \sum_{i, j} f^{i} (\nabla_{e_{j}} Z) f^{j} (\nabla_{e_{i}} Z)  \\
&+
\sum_{i} e^{i} (R(Z, e_{i})Z + \nabla_{e_{i}} \nabla_{Z} Z) 
+ \left( \sum_{i} e^{i} (\nabla_{e_{i}} Z) \right)^{2}
\\
&
\left. 
- \eta^{*}(\phi (\nabla_{Z} Z)) 
- g(Z, \phi Y) -2 \eta^{*}(\nabla_{Y} Z) 
\right \}\\
=&
- {\rm div} ({\rm div} (\rho_{\phi}[\iota] Y) Y)\\
&+\rho_{\phi}[\iota]  
\left \{
-2 \eta(Y) \sum_{i} e^{i} (\nabla_{e_{i}} Z) 
-2 \eta(Y) \sum_{i} f^{i} (\nabla_{e_{i}} Y) \right. \\
&
-2 f \sum_{i} e^{i} (\nabla_{e_{i}} Y) 
- {\rm Ric}(Y, Y)
+ \sum_{i} e^{i}(\nabla_{e_{i}}(\nabla_{Z} Z + \nabla_{Y} Y))\\
&
+ 
\left( \sum_{i} e^{i} (\nabla_{e_{i}} Z) \right)^{2}
+
\left( \sum_{i} e^{i} (\nabla_{e_{i}} Y) \right)^{2}\\
&
\left. 
- \eta^{*}(\phi (\nabla_{Z} Z)) + \eta^{*}(\phi (\nabla_{Y} Y))
-2 \eta^{*}(\nabla_{Y} Z)
+ \eta (Y)^{2}
\right \}. 
\end{align*}
From Corollary \ref{1st var cor}, we have 
\begin{align*}
\sum_{i} e^{i} (\nabla_{e_{i}} Z) = - g(Y, H_{\phi}) + \eta (Y), \qquad
\sum_{i} f^{i} (\nabla_{e_{i}} Y) = g(Y, H_{\phi}) - (n+1) \eta (Y), 
\end{align*}
which imply that
\begin{align} \label{comp 2nd var 1} 
-2 \eta (Y) \sum_{i} e^{i} (\nabla_{e_{i}} Z) 
-2 \eta (Y) \sum_{i} f^{i} (\nabla_{e_{i}} Y)
= 
2n \eta(Y)^{2}.
\end{align}
Since we know that 
\begin{align*}
\nabla_{Z} Z
&=
- \eta(Y) Z + \phi (\nabla_{Z} Y) + f Y + Z(f) \xi, \\
\nabla_{Y} Y
&=
- \phi \nabla_{Y} (\phi Y) + \eta (\nabla_{Y} Y) \xi - \eta(Y) \phi (Y), \\
\nabla_{Z} Z + \nabla_{Y} Y 
&=
\phi ([Z, Y]) -2 \eta(Y) \phi (Y) 
+2 f Y + (Z(f)+ \eta(\nabla_{Y} Y) -2 f \eta (Y)) \xi, 
\end{align*}
we see by Corollary \ref{1st var cor} that 
\begin{align*}
- \eta^{*}(\phi (\nabla_{Z} Z)) 
=
\eta (Y)^{2} - \eta (\pi_{L} (\nabla_{Z} Y)), \qquad
\eta^{*}(\phi (\nabla_{Y} Y))
= 
- \eta (\pi_{L} (\nabla_{Y} Z)) -  \eta (Y)^{2}.
\end{align*}
Using the equation
$\eta (\nabla_{Y} Z) = Y(f) + g(Y, Y) - \eta (Y)^{2}$, 
it follows that 
\begin{align} \label{comp 2nd var 2}
\begin{split}
- \eta^{*}(\phi (\nabla_{Z} Z)) + \eta^{*}(\phi (\nabla_{Y} Y))
-2 \eta^{*}(\nabla_{Y} Z)\\
=
\eta(\pi_{L} [Y, Z]) 
-2 Y(f) -2 g(Y,Y) +2 \eta(Y)^{2}.
\end{split}
\end{align}
We can also compute 
\begin{align} \label{comp 2nd var 3}
\begin{split}
\sum_{i} e^{i}(\nabla_{e_{i}}(\nabla_{Z} Z + \nabla_{Y} Y))
=& 
\frac{{\rm div}(\rho_{\phi}[\iota] (\pi_{L} \phi ([Z, Y]) + 2 f Y))}{\rho_{\phi}[\iota]}\\
&+
g(- \pi_{L} [Z, Y] + 2 \eta (Y) Y, H_{\phi}) \\
&+ \eta (\pi_{L} [Z, Y]) -2 \eta (Y)^{2}.
\end{split}
\end{align}
Note that 
\begin{align} \label{comp 2nd var 4}
f {\rm div} (\rho_{\phi}[\iota] Y) + \rho_{\phi}[\iota] Y(f)
=
{\rm div} (f \rho_{\phi}[\iota] Y). 
\end{align}
Then by 
(\ref{comp 2nd var 1}), (\ref{comp 2nd var 2}), (\ref{comp 2nd var 3}), 
(\ref{comp 2nd var 4})
and Corollary \ref{1st var cor}, 
we obtain  
\begin{align*}
\left.
\frac{\partial^{2}}{\partial t^{2}} \rho_{\phi}(t) \right|_{t=0}
=&
- {\rm div} ({\rm div} (\rho_{\phi}[\iota] Y) Y)
+ {\rm div}(\rho_{\phi}[\iota] (\pi_{L} \phi ([Z, Y]) + 2f Y))
-2 {\rm div} (f \rho_{\phi}[\iota] Y) 
\\
&+\rho_{\phi}[\iota]  
\left \{
(2n+2) \eta(Y)^{2} - {\rm Ric}(Y, Y) -2 g(Y, Y) \right. \\
&\left.
- g(\pi_{L} [Z, Y], H_{\phi})
+
g(Y, H_{\phi})^{2}
+
\frac{({\rm div}(\rho_{\phi}[\iota] Y))^{2}}{\rho_{\phi}[\iota]^{2}}
\right \},
\end{align*}
which implies Theorem \ref{2nd var phi Vol}.
\end{proof}

We investigate the relation 
of Theorem \ref{2nd var phi Vol} 
and the previous works. 
Define the standard Riemannian volume of $\iota$ by 
${\rm Vol}[\iota] = \int_{L} {\rm vol}_{\iota^{*}g}$.

\begin{rem} \label{relation Lmin}
We call a Legendrian immersion $\iota$ Legendrian-minimal Legendrian 
if it is a critical point of the standard Riemannian volume functional 
under Legendrian variations. 

Suppose that $\iota$ is Legendrian-minimal Legendrian 
and all of $\iota_{t}$'s are Legendrian in Theorem \ref{2nd var phi Vol}. 
Then for any $t$, the $\phi$-volume agrees with the standard Riemannian volume 
and the second variation formula of Theorem \ref{2nd var phi Vol} 
is given by 
\begin{align*}
\left. \frac{d^{2}}{dt^{2}} {\rm Vol}[\iota_{t}] \right|_{t=0} = &
\int_{L} 
\left(
\frac{1}{4} (\Delta f)^{2} - 2 g(Y, Y) - {\rm Ric}(\phi Y, \phi Y)  
 \right. \\
&
\left. 
-2 g(\nabla_{Y} Y, H)
+ g(Y, H)^{2}
\right)
{\rm vol}_{\iota^{*}g}, 
\end{align*}
where $H$ is the mean curvature vector of $L$ 
and $\Delta$ is the Laplacian acting on $C^{\infty}(L)$. 
This formula agrees with \cite[Theorem 1.1]{Kajigaya}. 
Thus 
when $\iota$ is minimal Legendrian and all of $\iota_{t}$'s are Legendrian, 
it agrees with \cite[Theorem 1.1]{Ono}.
\end{rem}

\begin{proof}
Since $\iota$ is Legendrian, we see that 
\begin{align} \label{Lmin comp 1}
\eta (Y)=0, \qquad 
\rho_{\phi}[\iota] = 1, \qquad
H_{\phi} = -\phi H, 
\end{align}
by Lemma \ref{equality rho phi} and Remark \ref{MC in Leg case}. 
Since all of $\iota_{t}$'s are Legendrian, we have $L_{Z} \eta = 2 g(Y, \cdot) + df =0$, 
which implies that 
\begin{align} \label{Lmin comp 2}
{\rm div}(Y) = \frac{1}{2} \Delta f.
\end{align}
A direct computation gives 
\begin{align} \label{Lmin comp 3}
\begin{split}
{\rm Ric}(Y, Y)            &= {\rm Ric}(\phi Y, \phi Y), \\
-g(\pi_{L}[Z, Y], H_{\phi}) &= g(\nabla_{Z} Y - \nabla_{Y} Z, \phi H), \\
g(\nabla_{Z} Y , \phi H) &= - g(\nabla_{Z}Z, H) \\
g(\nabla_{Y} Z, \phi H)  &= g(\nabla_{Y}Y, H).
\end{split}
\end{align}
By \cite[Lemma 4.1]{Kajigaya}, we have 
\begin{align} \label{Lmin comp 4}
\int_{L} g(\nabla_{Z} Z, H) {\rm vol}_{\iota^{*}g} 
=
\int_{L} g(\nabla_{Y} Y, H) {\rm vol}_{\iota^{*}g}, 
\end{align}
where we use 
an integration by parts argument 
and 
the Legendrian-minimality of $\iota$, 
which is equivalent to ${\rm div}(\phi H) = 0$ 
\cite[Theorem 3.6]{Kajigaya}. 
Then we obtain the statement by
(\ref{Lmin comp 1}), (\ref{Lmin comp 2}), (\ref{Lmin comp 3}) and (\ref{Lmin comp 4}). 
\end{proof}

Now we prove Theorems \ref{thm stablity}, \ref{thm convexity} and \ref{thm obstruction}. 

\begin{proof}[Proof of Theorem \ref{thm stablity}]
Let $\iota: L^{n}  \hookrightarrow M^{2n+1}$ be 
a $\phi$-minimal affine Legendrian submanifold. 
By definition, we have $H_{\phi} = 0$. 
Since $M^{2n+1}$ is a $\eta$-Einstein Sasakian manifold with the 
$\eta$-Ricci constant $A$, 
we see from Definition \ref{def of eta Einstein} that
\begin{align} \label{eta Einstein eq}
(2n+2) \eta (Y)^{2} - 2 g(Y, Y) - {\rm Ric}(Y, Y) 
= (A+2) \left( \eta(Y)^{2} - g(Y,Y) \right) 
\end{align}
for $Y \in TL$. 
By the third equation of Definition \ref{def of Sasakian 1}, 
we have $\eta(Y)^{2} - g(Y,Y) \leq 0$. 
Then Theorem \ref{2nd var phi Vol} 
implies Theorem \ref{thm stablity}. 
\end{proof}

\begin{proof}[Proof of Theorem \ref{thm convexity}]
Recall Lemma \ref{def of geodesic}. 
Let $\{ L_{t} \} \subset \mathcal{A}$ be a geodesic. 
Then there exists 
a curve 
of affine Legendrian embeddings $\{ \iota_{t} \}$, a fixed vector field $Y \in \mathfrak{X}(L)$
and a function $f \in C^{\infty}(L)$
such that 
$\pi (\iota_{t}) = L_{t}$, 
\begin{align} \label{Y Z commute}
\frac{d \iota_{t}}{dt} = \phi (\iota_{t})_{*} Y + f \xi \circ \iota_{t} \qquad \mbox{and} \qquad 
[ (\iota_{t})_{*} Y, \phi  (\iota_{t})_{*} Y + f \xi \circ \iota_{t}] = 0. 
\end{align}
Then 
by Theorem \ref{2nd var phi Vol}, (\ref{eta Einstein eq}) and $(\ref{Y Z commute})$, 
we obtain Theorem \ref{thm convexity}. 
\end{proof}

\begin{proof}[Proof of Theorem \ref{thm obstruction}]
Let $\iota: L^{n} \hookrightarrow M^{2n+1}$ be a $\phi$-minimal affine 
Legendrian submanifold.
Since $M^{2n+1}$ is a $\eta$-Einstein Sasakian manifold with the 
$\eta$-Ricci constant $A > -2$, 
we have 
\begin{align*}
(2n+2) \eta (Y)^{2} - 2 g(Y, Y) - {\rm Ric}(Y, Y) 
= (A+2) \left( \eta(Y)^{2} - g(Y,Y) \right) <0
\end{align*}
for any $0 \neq Y \in TL$ 
by the third equation of Definition \ref{def of Sasakian 1} and Definition \ref{def of aff Leg}. 
Take a 1-form 
$0 \neq\alpha\in\Omega^{1}(L)$ such that 
$d^{*} \alpha= 0$. 
For example, set 
$\alpha = d^{*} \beta$ for a 2-form $\beta$. 
Define the vector field $Y \in \mathfrak{X}(L)$ on $L$ via 
\begin{align*}
\iota^{*} g (\rho_{\phi}[\iota] Y, \cdot) = \alpha. 
\end{align*}
Then we easily see that ${\rm div}(\rho_{\phi}[\iota] Y) = - d^{*} \alpha = 0$. 
By Theorem \ref{2nd var phi Vol}, 
the second variation 
of the $\phi$-volume 
for this $Y$ is given by  
\begin{align*}
\left. \frac{d^{2}}{d t^{2}} 
\int_{L} {\rm vol}_{\phi} [\iota_{t}] \right|_{t=0} 
= 
\int_{L} 
\left(
 (A+2) \left( \eta(Y)^{2} - g(Y,Y) \right) 
\right)
{\rm vol}_{\phi}[\iota] < 0,
\end{align*}
which implies that $\iota: L \hookrightarrow M$ is not $\phi$-stable. 
\end{proof}

\section{$\phi$-volume in Sasaki-Einstein manifolds}

\subsection{$J$-volume in Calabi-Yau manifolds} \label{J-vol in CY}

\begin{definition}
Let $(X, h, J, \omega)$ be a real $2n$-dimensional K\"{a}hler manifold with a K\"{a}hler metric $h$, 
a complex structure $J$ and an associated K\"{a}hler form $\omega$.
Suppose that there exists 
a nowhere vanishing holomorphic $(n, 0)$-form $\Omega$ on $X$ satisfying 
\begin{align} \label{CYcondition}
\omega^{n}/n! = (-1)^{n(n-1)/2} (i/2)^{n} \Omega \wedge \bar{\Omega}, 
\end{align}
Then a quintuple $(X, h, J, \omega, \Omega)$  is called a
 {\bf Calabi-Yau manifold}. 
\end{definition}

\begin{rem}
The condition (\ref{CYcondition}) implies that 
$h$ is Ricci-flat and 
$\Omega$ is parallel with respect to the Levi-Civita connection of $h$.
\end{rem}

In Section \ref{J-vol in CY}, 
we suppose that $(X, h, J, \omega, \Omega)$  is a  real $2n$-dimensional Calabi-Yau manifold.

\subsubsection{Special Lagrangian geometry}

Define the {\bf Lagrangian angle} $\theta_{N}: N \rightarrow \mathbb{R}/ 2 \pi \mathbb{Z}$ 
of a Lagrangian immersion  $f: N \hookrightarrow X$ by
\begin{align*}
f^{*} \Omega = e^{i \theta_{N}} {\rm vol}_{f^{*}h}. 
\end{align*}
This is well-defined because 
$
|f^{*} \Omega (e_{1}, \cdots, e_{n})| = 1 
$
for any orthonormal basis $\{ e_{1}, \cdots, e_{n} \}$ of $T_{x}N$ for $x\in N$, 
which is proved in \cite[Theorem I\hspace{-.1em}I\hspace{-.1em}I.1.7]{HarveyLawson}. 
It also implies that 
Re$(e^{i \theta} \Omega)$ defines a calibration on $X$ for any $\theta \in \mathbb{R}$.

A Lagrangian immersion  $f: N \hookrightarrow X$
is called 
{\bf special Lagrangian}
if $f^{*} {\rm Re} \Omega = {\rm vol}_{f^{*}h}$, 
namely, the Lagrangian angle is $0$.

Minimal Lagrangian submanifolds 
are characterized 
in terms of special Lagrangian submanifolds 
as follows. 
For example, see \cite[Lemma 8.1]{LotayPacini}

\begin{lem} \label{equivalent min Lag}
Let $f: N \hookrightarrow X$ be an immersion 
of an oriented connected $n$-dimensional manifold $N$. 
The following are equivalent.
\begin{enumerate}[(a)]
\item $f^{*} {\rm Re}(e^{i \theta} \Omega) = {\rm vol}_{f^{*} h}$ for some $\theta \in \mathbb{R}$;
\item $f^{*} \omega = 0$ and $f^{*} {\rm Im}(e^{i \theta} \Omega) = 0$ for some $\theta \in \mathbb{R}$;
\item $f^{*} \omega = 0$ and the Lagrangian angle $\theta_{N}$ is constant;
\item $f: N \hookrightarrow X$ is minimal Lagrangian.
\end{enumerate}
\end{lem}

\subsubsection{Special affine Lagrangian geometry}

By using a $J$-volume, 
we can generalize the notion of calibrations. 

\begin{lem} [{\cite[Lemma 8.2]{LotayPacini}}] \label{J calibration}
Let $f: N \hookrightarrow X$ be an affine Lagrangian immersion 
of an oriented $n$-dimensional manifold $N$. 
Then we have  
\begin{align*}
f^{*} {\rm Re}\Omega \leq {\rm vol}_{J}[f] \leq {\rm vol}_{f^{*}h}.
\end{align*}
The equality holds 
\begin{itemize}
\item
in the first relation if and only if $f^{*} {\rm Im}\Omega =0$ and $f^{*} {\rm Re}\Omega > 0$, 
\item
and in the second relation if and only if $f$ is Lagrangian.
\end{itemize}
\end{lem}

Following \cite[Section 7.1]{LotayPacini}, 
define the {\bf affine Lagrangian angle} $\theta_{N}: N \rightarrow \mathbb{R}/ 2 \pi \mathbb{Z}$ 
of an affine Lagrangian immersion  $f: N \hookrightarrow X$ by
\begin{align*}
f^{*} \Omega = e^{i \theta_{N}} {\rm vol}_{J}[f]. 
\end{align*}
This is well-defined because 
\begin{align} \label{norm Omega rho J}
|f^{*} \Omega (e_{1}, \cdots, e_{n})| = \rho_{J}[f]
\end{align}
for any orthonormal basis $\{ e_{1}, \cdots, e_{n} \}$ of $T_{x}N$ for $x\in N$. 
The equation (\ref{norm Omega rho J})
is proved in \cite[Lemma 7.2]{LotayPacini}. 
We can also prove this directly by a pointwise calculation. 

We call an affine Lagrangian immersion $f: N \hookrightarrow X$
{\bf special affine Lagrangian}
if $f^{*} {\rm Re}\Omega = {\rm vol}_{J}[f] $, 
namely, the affine Lagrangian angle is $0$.

\begin{rem}
When $f: N \hookrightarrow X$ is Lagrangian, we have 
${\rm vol}_{J}[f] = {\rm vol}_{f^{*}h}$ by Lemma \ref{equality rho J}. 
Then 
the affine Lagrangian angle agrees with 
the standard Lagrangian angle. 
\end{rem}

We have an analogue of Lemma \ref{equivalent min Lag} 
given in \cite[Lemma 8.3]{LotayPacini}. 

\begin{lem} \label{equivalent J min Lag}
Let $f: N \hookrightarrow X$ be an affine Lagrangian immersion 
of an oriented connected real $n$-dimensional manifold $N$. 
The following are equivalent.
\begin{enumerate}[(a)]
\item $f^{*} {\rm Re}(e^{i \theta} \Omega) = {\rm vol}_{J}[f]$ for some $\theta \in \mathbb{R}$;
\item $f^{*} {\rm Im}(e^{i \theta} \Omega) = 0$ for some $\theta \in \mathbb{R}$;
\item the affine Lagrangian angle $\theta_{N}$ is constant;
\item $f: N \hookrightarrow X$ is a critical point for the $J$-volume.
\end{enumerate}
\end{lem}

\begin{proof}
Define $H_{J} \in C^{\infty}(N, J f_{*} TN)$ by 
\begin{align} \label{def of H_J}
H_{J} = -J ((J {\rm tr}_{N} (\pi_{J}^{t} \nabla^{X} \pi_{N}^{t}))^{\top}), 
\end{align}
where $\nabla^{X}$ is the Levi-Civita connection of $h$, 
$\top: \iota^{*} TX \rightarrow \iota_{*} TN$ is the tangential projection defined by $h$, 
and 
$\pi_{N}^{t}$ and $\pi_{J}^{t}$
are transposed operators of the canonical projections 
of $\pi_{N}: \iota^{*} TX \rightarrow \iota_{*} TN $ and 
$\pi_{J}: \iota^{*} TX \rightarrow J \iota_{*} TN$
via the decomposition $\iota^{*} TX = \iota_{*} TN \oplus J \iota_{*} TN$
, respectively. 

By \cite[Proposition 5.2]{LotayPacini}, 
$f: N \hookrightarrow X$ is a critical point for the $J$-volume if and only if 
$H_{J} = 0$.
By \cite[Corollary 7.4]{LotayPacini}, we have 
\begin{align} \label{H_J angle}
H_{J} = J (d \theta_{N})^{\sharp},
\end{align}
where $\sharp$ is the metric dual with respect to $f^{*} h$.
Then by (\ref{norm Omega rho J}) and (\ref{H_J angle}), 
we see the equivalence.
\end{proof}

\subsection{$\phi$-volume in Sasaki-Einstein manifolds} \label{phi-vol in SE}

The odd dimensional analogue of a Calabi-Yau manifold
is a Sasaki-Einstein manifold. 
The following is a well-known fact. 
For example, see \cite[Lemma 11.1.5]{BoyerGalicki}.

\begin{lem}
Let $(M, g, \eta, \xi, \phi)$ be a $(2n+1)$-dimensional Sasakian manifold. 
If $g$ is Einstein, a cone $(C(M), \bar{g})$ is Ricci-flat. 
\end{lem}

Thus the canonical bundle of $C(M)$ is diffeomorphically trivial. 
In addition, 
suppose that 
the cone $C(M)$ is a Calabi-Yau manifold, 
namely, 
there exists 
a nowhere vanishing holomorphic $(n+1, 0)$-form $\Omega$ on $C(M)$ such that 
\begin{align} \label{CYcondition cone}
\bar{\omega}^{n+1}/(n+1)! = (-1)^{n(n+1)/2} (i/2)^{n+1} \Omega \wedge \overline{\Omega}, 
\end{align}
where $\bar{\omega} = \bar{g}(J \cdot, \cdot)$ is the associated K\"ahler form on $C(M)$. 
Then 
the canonical bundle of $C(M)$ is holomorphically trivial.

\begin{lem}[{\cite[Corollary 11.1.8]{BoyerGalicki}}]
If $M$ is a compact simply-connected Sasaki-Einstein manifold, 
$C(M)$ is a Calabi-Yau manifold. 
\end{lem}

\begin{rem}
The holomorphic volume form $\Omega$ is not unique. 
For any $\theta \in \mathbb{R}$, 
$e^{i \theta} \Omega$ also satisfies (\ref{CYcondition cone}). 
\end{rem}

In Section \ref{phi-vol in SE}, we suppose that 
$M$ is a $(2n+1)$-dimensional Sasaki-Einstein manifold with 
a Calabi-Yau structure $(\bar{g}, J, \bar{\omega}, \Omega)$ on $C(M)$.

Define a complex valued $n$-form on $M$ by 
\begin{align*}
\psi = u^{*} \left( i \left(r \frac{\partial}{\partial r} \right) \Omega \right),
\end{align*}
where $u: M =\{ 1 \} \times M \hookrightarrow C(M)$ is an inclusion. 
Note that we can recover $\Omega$ from $\psi$ via
\begin{align} \label{Omega psi}
\Omega = (dr - i \eta) \wedge r^{n} \psi.
\end{align}

\subsubsection{Special Legendrian geometry}

Define the {\bf Legendrian angle} $\theta_{L}: L \rightarrow \mathbb{R}/ 2 \pi \mathbb{Z}$ 
of a Legendrian immersion  $\iota: L \hookrightarrow M$ by
\begin{align*}
\iota^{*} \psi = e^{i \theta_{L}} {\rm vol}_{\iota^{*} g}. 
\end{align*}
Note that 
the Legendrian angle $\theta_{L}$ of a Legendrian immersion  $\iota: L \hookrightarrow M$  
agrees with the Lagrangian angle 
of the induced Lagrangian immersion $\bar{\iota}:  C(L) \hookrightarrow C(M)$ given by (\ref{induced imm}).

\begin{definition}
Let $L$ be an oriented $n$-dimensional manifold 
admitting an immersion $\iota: L \hookrightarrow M$. 
An immersion $\iota: L \hookrightarrow M$ is called {\bf special Legendrian} if 
$\iota^{*} {\rm Re} \psi = {\rm vol}_{\iota^{*}g}$.
This is equivalent to the condition that
the induced immersion $\bar{\iota}:  C(L) \hookrightarrow C(M)$ given by (\ref{induced imm}) 
is special Lagrangian. 
\end{definition}

We have an analogue of Lemmas \ref{equivalent min Lag} and \ref{equivalent J min Lag}. 
This is a direct consequence of Lemma \ref{equivalent min Lag}. 
Note that $\iota$ is minimal if and only if $\bar{\iota}$ is minimal. 

\begin{lem} \label{equivalent min Leg}
Let $\iota: L \hookrightarrow M$ be an immersion 
of an oriented connected $n$-dimensional manifold $L$. 
The following are equivalent.
\begin{enumerate}[(a)]
\item $\iota^{*} {\rm Re}(e^{i \theta} \psi) = {\rm vol}_{\iota^{*}g}$ for some $\theta \in \mathbb{R}$;
\item $\iota^{*} \eta = 0$ and $\iota^{*} {\rm Im}(e^{i \theta} \psi) = 0$ for some $\theta \in \mathbb{R}$;
\item $\iota^{*} \eta = 0$ and the Legendrian angle $\theta_{L}$ is constant;
\item $\iota: L \hookrightarrow M$ is minimal Legendrian. 
\end{enumerate}
\end{lem}

\subsubsection{Special affine Legendrian geometry}

From Lemma \ref{J calibration}, we immediately see the following. 

\begin{lem} \label{phi calibration}
Let $\iota: L \hookrightarrow M$ be an affine Legendrian immersion 
of an oriented $n$-dimensional manifold $L$. 
Then we have  
\begin{align*}
\iota^{*} {\rm Re}\psi \leq {\rm vol}_{\phi}[\iota] \leq {\rm vol}_{\iota^{*}g}.
\end{align*}
The equality holds 
\begin{itemize}
\item
in the first relation if and only if $\iota^{*} {\rm Im}\psi =0$ and $\iota^{*} {\rm Re}\psi > 0$, 
\item
and in the second relation if and only if $\iota$ is Legendrian.
\end{itemize}
\end{lem}

Define the {\bf affine Legendrian angle} $\theta_{L}: L \rightarrow \mathbb{R}/ 2 \pi \mathbb{Z}$ 
of an affine Legendrian immersion  $\iota: L \hookrightarrow M$ by
\begin{align*}
\iota^{*} \psi = e^{i \theta_{L}} {\rm vol}_{\phi}[\iota]. 
\end{align*}
This is well-defined by (\ref{norm Omega rho J}).
Note that 
the affine Legendrian angle $\theta_{L}$ of an affine Legendrian immersion $\iota: L \hookrightarrow M$  
agrees with the affine Lagrangian angle 
of the induced affine Lagrangian immersion 
$\bar{\iota}:  C(L) \hookrightarrow C(M)$ given by (\ref{induced imm}).

\begin{definition} \label{def of special aff Leg}
Let $L$ be an oriented $n$-dimensional manifold 
admitting an immersion $\iota: L \hookrightarrow M$. 
An immersion $\iota: L \hookrightarrow M$ is called {\bf special  affine Legendrian} if 
$\iota^{*} {\rm Re} \psi = {\rm vol}_{\phi}[\iota]$.
This is equivalent to the condition that
the induced immersion $\bar{\iota}:  C(L) \hookrightarrow C(M)$ given by (\ref{induced imm}) 
is special affine Lagrangian. 
\end{definition}

It is natural to expect an analogue of Lemmas \ref{equivalent min Lag}, 
\ref{equivalent J min Lag} and \ref{equivalent min Leg}. 
By Lemma \ref{phi calibration}, 
we immediately see that the following three conditions are equivalent. 
\begin{itemize}
\item $\iota^{*} {\rm Re}(e^{i \theta} \psi) = {\rm vol}_{\phi}[\iota]$ for some $\theta \in \mathbb{R}$;
\item $\iota^{*} {\rm Im}(e^{i \theta} \psi) = 0$ for some $\theta \in \mathbb{R}$;
\item the affine Legendrian angle $\theta_{L}$ is constant.
\end{itemize}
However, 
in the affine Legendrian setting, 
each of these conditions is not equivalent to saying that 
$\iota$ is a critical point for the $\phi$-volume. 
In fact, we have the following.

\begin{prop} \label{relation angle H phi}
Let $\iota: L \hookrightarrow M$ be an affine Legendrian immersion 
of an oriented connected $n$-dimensional manifold $L$. 
We have 
\begin{align*}
(d \theta_{L})^{\sharp} = - (n+1) \xi^{\top} + H_{\phi}, 
\end{align*}
where $\sharp$ is the metric dual with respect to $\iota^{*} g$ on $L$,
$\top: \iota^{*} TM \rightarrow TL$ is the tangential projection defined by 
the orthogonal decomposition of $\iota^{*} TM$ by the metric $g$
and $H_{\phi}$ is given in Definition \ref{def of H phi}. 
\end{prop}

Thus an analogue of Lemmas \ref{equivalent min Lag}, 
\ref{equivalent J min Lag} and \ref{equivalent min Leg} 
holds if 
an affine Legendrian immersion
$\iota: L \hookrightarrow M$ is Legendrian. 
It may be necessary to modify the notion of the $\phi$-volume 
to hold an analogue of Lemmas \ref{equivalent min Lag}, 
\ref{equivalent J min Lag} and \ref{equivalent min Leg} 
or 
to consider 
what the critical points of the $\phi$-volume which are not minimal Legendrian are.

\begin{lem} \label{H_J H_phi}
Let $H_{J} \in C^{\infty}(C(L), J \bar{\iota}_{*} TC(L))$ 
be defined by (\ref{def of H_J}). Then we have 
\begin{align*}
(J H_{J})|_{r=1} = (n+1) \xi^{\top} - H_{\phi}.
\end{align*}
\end{lem}

\begin{proof}
By (\ref{def of H_J}), 
we have for any vector field $Y$ on $C(L)$ 
\begin{align*}
\bar{\iota}^{*} \bar{g} (Y, J H_{J}) &= 
\sum_{i=1}^{n} \bar{\iota}^{*} \bar{g} 
\left( \pi_{C(L)} (\bar{\nabla}_{\frac{e_{i}}{r}} (JY)), \frac{e_{i}}{r} \right)
+ \bar{\iota}^{*} \bar{g} \left (\pi_{C(L)} (\bar{\nabla}_{\frac{\partial}{\partial r}} (JY)), 
\frac{\partial}{\partial r} \right) \\
&=
\sum_{i=1}^{n} \bar{\iota}^{*} \bar{g} 
\left( \pi_{C(L)} (\bar{\nabla}_{\frac{e_{i}}{r}} (JY)), \frac{e_{i}}{r} \right),
\end{align*}
where $\{ e_{1}, \cdots, e_{n} \}$ is a local orthonormal frame of $TL$ 
with respect to $\iota^{*}g$. 
Using the notation in Section \ref{first var}, 
we have for any vector field $Y$ on $L$
\begin{align*}
\bar{\nabla}_{e_{i}} (JY) &= 
J \left(\nabla_{e_{i}} Y - \frac{\bar{g}(e_{i}, Y)}{r^{2}} \cdot r \frac{\partial}{\partial r} \right),\\
\nabla_{e_{i}} Y &= 
\sum_{j=1}^{n} e^{j} (\nabla_{e_{i}} Y) e_{j} + \sum_{j=1}^{n} f^{j} (\nabla_{e_{i}} Y) \phi e_{j} 
+ \eta^{*} (\nabla_{e_{i}} Y) \xi \\
&=
\sum_{j=1}^{n} e^{j} (\nabla_{e_{i}} Y) e_{j} 
+ \sum_{j=1}^{n} f^{j} (\nabla_{e_{i}} Y) 
\left (J e_{j} - \eta (e_{j}) r \frac{\partial}{\partial r} \right)
+ \eta^{*} (\nabla_{e_{i}} Y) \xi.
\end{align*}
Hence we have at the point of $\{ r=1 \}$
\begin{align*}
\sum_{i=1}^{n} \bar{\iota}^{*} \bar{g} (\pi_{C(L)} (\bar{\nabla}_{e_{i}} (JY)), e_{i})
&=
- \sum_{i=1}^{n} f^{i} (\nabla_{e_{i}} Y)\\
&=
\sum_{i=1}^{n} e^{i} (\phi (\nabla_{e_{i}} Y))\\
&=
\sum_{i=1}^{n} e^{i} (\nabla_{e_{i}} (\phi Y) - (\nabla_{e_{i}} \phi) (Y))\\
&=
\sum_{i=1}^{n} e^{i} (\nabla_{e_{i}} (\phi Y)) + n \eta(Y). 
\end{align*}
Since we know that  
$\sum_{i=1}^{n} e^{i} (\nabla_{e_{i}} (\phi Y)) = - g (Y, H_{\phi}) + \eta (Y)$ 
by Corollary \ref{1st var cor}, 
the proof is done.
\end{proof}

\begin{proof}[Proof of Proposition \ref{relation angle H phi}]
By (\ref{H_J angle}), we have 
\begin{align*}
d \theta_{C(L)} = - \bar{\iota}^{*} \bar{g} (J H_{J}, \cdot), 
\end{align*}
where $\theta_{C(L)}$ is the affine Lagrangian angle of 
$\bar{\iota}:  C(L) \hookrightarrow C(M)$ given by (\ref{induced imm}). 
Since we know that 
$u^{*} \theta_{C(L)} = \theta_{L}$ for the inclusion $u: L \hookrightarrow C(L)$, 
Lemma \ref{H_J H_phi} implies the statement. 
\end{proof}

\begin{rem}
We can also prove Proposition \ref{relation angle H phi} by using the tensors on $L$. 
We give an outline of the proof. 
Define the $1$-form $\xi_{\phi}$ on $L$ by $u^{*} \xi_{J}$: 
the pullback of the Maslov form $\xi_{J}$ of $C(L)$ defined in \cite[Section 3.2]{LotayPacini}
by $u: L \hookrightarrow C(L)$. 
Define the complex valued $n$-form $\psi_{L}$ on $L$ by 
$\psi_{L} = u^{*} (i(\partial/ \partial r) \Omega_{C(L)})$, 
where $\Omega_{C(L)}$ is the canonical section 
of the canonical bundle of $C(L)$ 
defined in \cite[Section 3.2]{LotayPacini}.

Then as in \cite[Lemma 7.2]{LotayPacini}, we have $\psi = e^{i \theta_{L}} \psi_{L}$.
By a direct computation, we have 
\begin{align*}
\xi_{\phi} = - \sum_{i=1}^{n} f^{i} (\nabla e_{i}), \qquad 
\xi_{\phi}^{\sharp}= (n+1) \xi^{\top} - H_{\phi}. 
\end{align*}
Let $\bar{\nabla}$ and $\nabla$ be the Levi-Civita connections of 
$\bar{g}$ and $g$, respectively. 
By the equations $\bar{\nabla} \Omega = 0$ and (\ref{Omega psi}), we deduce that 
\begin{align*}
\nabla \psi = - i \eta \wedge \psi. 
\end{align*}
By the equations $\bar{\nabla} \Omega_{L} = i \xi_{J} \otimes \Omega_{C(L)}$ 
and  $\Omega_{C(L)} = (dr - i \eta) \wedge r^{n} \psi_{L}$, 
we deduce that 
\begin{align*}
\nabla \psi_{L} = i (-\eta \wedge \psi_{L} + \xi_{\phi} \otimes \psi_{L}).
\end{align*}
Then we obtain $\xi_{\phi} = - d \theta_{L}$ from $\psi = e^{i \theta_{L}} \psi_{L}$.
\end{rem}

\section{Moduli space of the special affine Legendrian submanifolds}

In this section, we prove Theorem \ref{smooth moduli affine Leg}. 
First, we study the moduli space of submanifolds 
characterized by differential forms 
following \cite{Moriyama} 
to obtain Proposition \ref{general moduli}. 
As a corollary of  Proposition \ref{general moduli}, 
we prove Theorem \ref{smooth moduli affine Leg}.

Let $(M, g)$ be a Riemannian manifold 
and $L$ be a compact connected manifold 
admitting an embedding into $M$. 
Denote by $C^{\infty}_{emb}(L, M)$ be 
the set of all embeddings from $L$ to $M$:
\begin{align*}
C^{\infty}_{emb}(L, M) = \{ \iota: L \hookrightarrow M; \iota \mbox{ is an embedding} \}.
\end{align*}
Set $\mathcal{M}(L, M) = C^{\infty}_{emb}(L, M)/ {\rm Diff}^{\infty}(L)$, 
where ${\rm Diff}^{\infty}(L)$ is a $C^{\infty}$ diffeomorphism group of $L$.

By \cite[Theorem 3.3]{Opozda},  
$\mathcal{M}(L, M)$ is a smooth Fr\'{e}chet manifold 
modeled on the Fr\'{e}chet vector space 
$C^{\infty}(L, \mathcal{N}_{\iota})$ for $\iota \in C^{\infty}_{emb}(L, M)$, 
where 
$\mathcal{N}_{\iota}$ is any vector bundle transversal to $\iota$
and 
$C^{\infty}(L, \mathcal{N}_{\iota})$ 
is the space of all sections of $\mathcal{N}_{\iota} \rightarrow L$.

Now we choose
a system 
$\Phi = (\varphi_{1}, \cdots, \varphi_{m}) \in \oplus_{i=1}^{m} \Omega^{k_{i}} (M)$
of smooth differential forms on $M$. 
These forms are not necessarily closed. 

\begin{definition} \label{def of phi embed}
The embedding $\iota \in C^{\infty}_{emb}(L, M)$ 
is called a {\bf $\Phi$-embedding} if 
\begin{align*}
\iota^{*} \Phi = (\iota^{*} \varphi_{1}, \cdots, \iota^{*} \varphi_{m}) = 0.
\end{align*}
Define the moduli space 
$\mathcal{M}_{L} (\Phi)$
of 
$\Phi$-embeddings of $L$ by
\begin{align*}
\mathcal{M}_{L} (\Phi)
=
\{ \iota \in C^{\infty}_{emb}(L, M); \iota^{*} \Phi =0 \}/ {\rm Diff}^{\infty}(L). 
\end{align*}
\end{definition}
We want to study the structure of $\mathcal{M}_{L} (\Phi)$. 

Fix $\iota \in C^{\infty}_{emb}(L, M)$ 
satisfying $\iota^{*} \Phi = 0$
and 
a vector bundle $\mathcal{N}_{\iota} \rightarrow L$ which is transversal to $\iota$. 
Set 
\begin{align*}
V_{1} = C^{\infty}(L, \mathcal{N}_{\iota}), \qquad
V_{2} = \oplus_{i=1}^{m} \Omega^{k_{i}} (L) = C^{\infty} (L, \oplus_{i=1}^{m} \wedge^{k_{i}} T^{*}L).
\end{align*}
By the tubular neighborhood theorem there exists a neighborhood
of $L$ in $M$ which is identified with 
an open neighborhood $\mathcal{U} \subset \mathcal{N}_{\iota}$ of the zero section by
the exponential map. 
Set
\begin{align*}
U = \{ v \in V_{1}; v_{x} \in \mathcal{U} \mbox{ for any } x \in L \}.
\end{align*}
The exponential map induces the embedding 
${\rm exp}_{v}: L \hookrightarrow M$
by ${\rm exp}_{v}(x) = {\rm exp}_{x} (v_{x})$ 
for $v \in U$ and $x \in L$. 
Define the first order differential operator $F: U \rightarrow V_{2}$ by
\begin{align*}
F(v) = 
{\rm exp}_{v}^{*} \Phi
=
({\rm exp}_{v}^{*} \varphi_{1}, \cdots, {\rm exp}_{v}^{*} \varphi_{m}).
\end{align*} 
Then 
${\rm exp}_{v}: L \hookrightarrow M$ is $\Phi$-embedding 
if and only if $F(v)=0$. 
Thus 
a neighborhood of $[\iota]$ in $\mathcal{M}_{L} (\Phi)$
is identified 
with that of $0$ in $F^{-1}(0)$ (in the $C^{1}$ sense). 
Let $D_{1}$ be the
linearization of $F$ at $0$:
\begin{align*}
D_{1} = (dF)_{0}: V_{1} \rightarrow V_{2}. 
\end{align*}

First, we prove the following, 
which is a slight generalization of \cite[Proposition 2.2]{Moriyama}. 
It will be useful to see 
whether the moduli space of submanifolds 
characterized by some differential forms 
is smooth.
We use the notion of a Fr\'{e}chet manifold given in \cite{Hamilton}.

\begin{prop} \label{general moduli}
Suppose that there exist a vector bundle $E \rightarrow L$ 
and a first order differential operator $D_{2}: V_{2} \rightarrow V_{3}$, 
where $V_{3} = C^{\infty}(L, E)$ is a space of 
smooth sections of $E \rightarrow L$, 
such that 
\begin{align*}
V_{1} \overset{D_{1}}{\longrightarrow} V_{2} \overset{D_{2}}{\longrightarrow} V_{3}
\end{align*}
is a differential complex. 
Namely, $D_{2} \circ D_{1} = 0$. 
Denote by $D_{i}^{*} : V_{i+1} \rightarrow V_{i}$
the formal adjoint operator of $D_{i}$. 

\begin{enumerate}
\item
Suppose that $P_{2} = D_{1} D_{1}^{*} + D_{2}^{*} D_{2}: V_{2} \rightarrow V_{2}$
is elliptic and
${\rm Im}(F) \subset {\rm Im}(D_{1})$. 
Then 
around $[\iota]$, 
the moduli space 
$\mathcal{M}_{L} (\Phi)$ is a smooth Fr\'{e}chet manifold 
and it is a submanifold of $\mathcal{M} (L, M)$.

\item
In addition to the assumptions of 1, 
suppose further that 
$P_{1} = D_{1}^{*} D_{1}: V_{1} \rightarrow V_{1}$ is elliptic. 
Then the moduli space $\mathcal{M}_{L} (\Phi)$ is 
a finite dimensional smooth manifold around $[\iota]$ 
and its dimension is equal to $\dim \ker (D_{1})$. 
\end{enumerate}
\end{prop}

\begin{proof}
Consider the case 1. 
First, we extend above spaces and operators to those of class $C^{k ,a}$, 
where $k \geq 1$ is an integer and $0 < a <1$. Set 
\begin{align*}
V_{1}^{k, a} &= C^{k, a} (L, \mathcal{N}_{\iota}), \qquad
V_{2}^{k, a} = C^{k, a} (L, \oplus_{i=1}^{m} \wedge^{k_{i}} T^{*}L), \\
U^{k, a} &= \{ v \in V_{1}; v_{x} \in \mathcal{U} \mbox{ for any } x \in L \},\\
F^{k, a}&: U^{k, a} \rightarrow V_{2}^{k-1, a}, \qquad
D_{1}^{k, a} = (dF^{k, a})_{0}, \\
\mathcal{M}_{L}^{k, a} (\Phi)
&=
\{ \iota \in C^{k, a}_{emb}(L, M); \iota^{*} \Phi =0 \}/ {\rm Diff}^{k, a}(L). 
\end{align*}
Similarly, 
a neighborhood of $[\iota]$ in $\mathcal{M}^{k, a}_{L} (\Phi)$
is identified 
with that of $0$ in $(F^{k, a})^{-1}(0)$ 
(in the $C^{1}$ sense).

We prove that 
$\mathcal{M}^{k, a}_{L} (\Phi)$ 
is smooth around $[\iota]$ 
in the sense of Banach. 
To apply the implicit function theorem, 
we prove the following. 

\begin{lem} \label{prepare for imp fn}
\begin{enumerate}[(a)]
\item ${\rm Im}(D_{1}^{k,a}) \subset V_{2}^{k-1, a}$ is a closed subspace.
\item ${\rm Im}(F^{k,a}) \subset {\rm Im}(D_{1}^{k,a})$.
\item $D_{1}^{k,a}: V_{1}^{k,a} \rightarrow {\rm Im}(D_{1}^{k,a})$ has a right inverse. 
\end{enumerate}
\end{lem}

\begin{proof}
By the Hodge decomposition, we have 
\begin{align} \label{Hodge decomp}
V_{2}^{k-1, a} = \ker P_{2} \oplus D_{1}^{k,a} (V_{1}^{k,a}) \oplus (D_{2}^{*})^{k,a}(V_{3}^{k,a}),
\end{align}
where $(D_{2}^{*})^{k,a}: V_{3}^{k,a} \rightarrow V_{2}^{k-1,a}$ is a canonical extension of $D_{2}^{*}$. 
This is a $L_{2}$-orthogonal decomposition 
and 
${\rm Im}(D_{1}^{k,a})$ is the orthogonal complement of 
$\ker P_{2} \oplus (D_{2}^{*})^{k,a}(V_{3}^{k,a})$. Thus we see (a).

We prove (b).
For any $f \in U^{k, a}$, there exists a sequence $\{ f_{n} \} \subset U$ 
such that $f_{n} \rightarrow f$. 
By 
$F (f_{n}) \in {\rm Im}(F) \subset {\rm Im}(D_{1}) \subset {\rm Im}(D_{1}^{k,a})$, 
$F^{k,a}(f) = \lim_{n \rightarrow \infty} F(f_{n})$ and (a), 
we see that 
$F^{k,a}(f) \in {\rm Im}(D_{1}^{k,a})$. 

We prove (c).
Let $G$ be the Green's operator of $P_{2}$. 
Then for any $f \in {\rm Im}(D_{1}^{k,a})$, we have 
$f = P_{2} G(f) = D_{1}^{k,a} (D_{1}^{*})^{k+1,a} G(f) + (D_{2}^{*})^{k,a} D_{2}^{k+1,a} G(f)$. 
By (\ref{Hodge decomp}), we deduce that 
\begin{align}\label{eq Im D1}
D_{1}^{k,a} (D_{1}^{*})^{k+1,a} G(f) = f, \qquad
(D_{2}^{*})^{k,a} D_{2}^{k+1,a} G(f) = 0.
\end{align}
Thus we see that 
$(D_{1}^{*})^{k+1,a} G|_{{\rm Im}(D_{1}^{k,a})} : {\rm Im}(D_{1}^{k,a}) \rightarrow V_{1}^{k,a}$ 
is a right inverse of $D_{1}^{k,a}: V_{1}^{k,a} \rightarrow {\rm Im}(D_{1}^{k,a})$.
\end{proof}

By Lemma \ref{prepare for imp fn} (a), 
we obtain the smooth map 
$F^{k,a}: U^{k,a} \rightarrow {\rm Im}(D_{1}^{k,a})$. 
The smoothness of this map is proved in \cite[Theorem 2.2.15]{Baier}. 
It is clear that 
$(d F^{k,a})_{0} = D_{1}^{k,a}: V_{1}^{k,a} \rightarrow {\rm Im}(D_{1}^{k,a})$
is surjective 
and $V_{1}^{k,a}$ is the direct sum of the kernel of $D_{1}^{k,a}$ 
and the image of the right inverse of $D_{1}^{k,a}: V_{1}^{k,a} \rightarrow {\rm Im}(D_{1}^{k,a})$. 
By the proof of Lemma \ref{prepare for imp fn} (c), we have 
\begin{align*}
V_{1}^{k,a} =  X_{1}^{k,a} \oplus Y_{1}^{k,a}, 
\end{align*}
where
$X_{1}^{k,a} = \ker (D_{1}^{k,a})$ and 
$Y_{1}^{k,a} = (D_{1}^{*})^{k+1,a} G {\rm Im}(D_{1}^{k,a}).$ 
Note that 
both spaces are closed in $V_{1}^{k,a}$.

Then we can apply the implicit function theorem. 
There exist 
an open neighborhood $A_{1}^{k,a} \subset X_{1}^{k,a}$ of $0$, 
an open neighborhood $B_{1}^{k,a} \subset Y_{1}^{k,a}$ of $0$, 
and a smooth mapping 
$\hat{G}^{k,a}: A_{1}^{k,a} \rightarrow B_{1}^{k,a}$ such that 
\begin{align*}
(F^{k,a})^{-1}(0) \cap (A_{1}^{k,a} \oplus B_{1}^{k,a}) 
= 
\{ x + \hat{G}^{k,a}(x);  x \in A_{1}^{k,a} \},
\end{align*}
which implies that $\mathcal{M}_{L}^{k,a} (\Phi)$ is smooth around $[\iota]$
in the sense of Banach.

Next, we prove that 
$\mathcal{M}_{L} (\Phi)$ 
is smooth 
around $[\iota]$ in the sense of Fr\'{e}chet. 
The proof is an analogue of that of 
\cite[Theorem 4.1]{Opozda}. 
The open set $A_{1}^{k,a}$ 
and the map $\hat{G}^{k,a}$
depend on $k$ and $a$.
We have to show that 
we can take $A_{1}^{k,a}$ and $\hat{G}^{k,a}$
``uniformly". 
Namely, set 
\begin{align*}
G^{k,a} = \hat{G}^{1,a}|_{A_{1}^{1,a} \cap V_{1}^{k,a}}: 
A_{1}^{1,a} \cap V_{1}^{k,a} \rightarrow B_{1}^{1,a}.
\end{align*}

In the following, 
by shrinking $A_{1}^{1,a}$ if necessary, 
we prove that for any $k \geq 1$
\begin{itemize}
\item ${\rm Im}(G^{k,a}) \subset Y_{1}^{k,a} = Y_{1}^{1,a} \cap V_{1}^{k,a}$, 
\item and $G^{k,a}: A_{1}^{1,a} \cap V_{1}^{k,a} \rightarrow Y_{1}^{k,a}$ is smooth in the sense of Banach. 
\end{itemize}
Then we see that 
${\rm Im}(\hat{G}^{1,a}|_{A_{1}^{1,a} \cap V_{1}}) \subset Y_{1}^{1,a} \cap V_{1}$ 
and 
$\hat{G}^{1,a}|_{A_{1}^{1,a} \cap V_{1}}$ 
is smooth in the sense of Fr\'{e}chet. 
Hence we see that $\mathcal{M}_{L} (\Phi)$ 
is smooth around $[\iota]$.

First,  we show that 
${\rm Im} (G^{k,a}) \subset Y_{1}^{k,a}$ 
by the elliptic regularity theorem. 
For any $\gamma \in V_{1}$, define 
the second order differential operator $F_{\gamma}: V_{2} \rightarrow V_{2}$ by 
\begin{align*}
F_{\gamma} (\beta) = F(\gamma + D_{1}^{*} \beta) + D_{2}^{*} D_{2} \beta.
\end{align*}
Denote by 
$F_{\gamma}^{1,a}$ 
the extension of $F_{\gamma}$ on $V_{2}^{1,a}$.

Since the linearization of $F_{0}$ at $0$, 
which is given by 
$(dF_{0})_{0} = D_{1} D_{1}^{*} + D_{2}^{*} D_{2} = P_{2}$, is elliptic
and 
the ellipticity is an open condition, we see that 
there exist 
an open neighborhood $\mathcal{U}_{0} \subset V_{1}^{1,a}$ of $0$ 
and an open neighborhood $\mathcal{V}_{0} \subset V_{2}^{2,a}$ of $0$ 
such that 
$(dF^{1,a}_{\gamma})_{\beta}$ is elliptic 
for any $(\gamma, \beta) \in \mathcal{U}_{0} \times \mathcal{V}_{0}.$
Set 
\begin{align*}
\mathcal{U}_{1} = (G^{1,a})^{-1}( (D_{1}^{*})^{2,a} (\mathcal{V}_{0} \cap G ({\rm Im}(D_{1}^{1,a})))
\cap B_{1}^{1,a}) \cap \mathcal{U}_{0}, 
\end{align*}
which is an open subset of $A_{1}^{1,a}$ because 
\begin{align*}
(D_{1}^{*})^{2,a}|_{G ({\rm Im} (D_{1}^{1,a}))}: G ({\rm Im} (D_{1}^{1,a})) \rightarrow 
(D_{1}^{*})^{2,a} G ({\rm Im} (D_{1}^{1,a})) = Y_{1}^{1,a}
\end{align*}
is an isomorphism.

\begin{lem} \label{image Ck}
For any $k \geq 1$, we have 
\begin{align*}
G^{1,a} (\mathcal{U}_{1} \cap V_{1}^{k,a}) \subset Y_{1}^{k,a}.
\end{align*}
\end{lem}

\begin{proof}
Let $\alpha \in \mathcal{U}_{1} \cap V_{1}^{k,a}$. 
Since $F$ is the first order differential operator, 
the differential operator $F_{\alpha}$ is of class $C^{k-1,a}$.
By the definition of $\mathcal{U}_{1}$, 
there exists $\beta \in \mathcal{V}_{0} \cap G ({\rm Im}(D_{1}^{1,a}))$ 
satisfying 
$G^{1,a}(\alpha) = (D_{1}^{*})^{2,a} (\beta)$. Then
\begin{align*}
F^{1,a}_{\alpha} (\beta) = F^{1,a} (\alpha + (D_{1}^{*})^{2,a} \beta) + 
(D_{2}^{*})^{1,a} D_{2}^{2,a} \beta = 0
\end{align*}
by the definition of $G^{1,a}$ and (\ref{eq Im D1}). 
Since $(\alpha, \beta) \in \mathcal{U}_{0} \times \mathcal{V}_{0}$, 
$(dF^{1,a}_{\alpha})_{\beta}$ is elliptic.
Hence Schauder theory implies that 
$\beta$ is of class $C^{k+1, a}$. 
Thus  
$G^{1,a}(\alpha) = (D_{1}^{*})^{2,a} (\beta)$ is of class $C^{k,a}$. 
\end{proof}

Next, we show that $G^{k,a}$ is a smooth map. 
Since 
$(d F^{1,a})_{0}|_{Y_{1}^{1,a}} = D_{1}^{1,a}|_{Y_{1}^{1,a}} : Y_{1}^{1,a} 
\rightarrow Z_{2}^{0,a} = D_{1}^{1,a} (V_{1}^{1,a})$
is an isomorphism 
and being an isomorphism is an open condition, 
there is an open neighborhood $\mathcal{U}_{2} \subset V_{1}^{1,a}$ of $0$ such that 
$
(d F^{1,a})_{\gamma}|_{Y_{1}^{1,a}} 
:Y_{1}^{1,a} \rightarrow Z_{2}^{0,a}
$
is an isomorphism for any $\gamma \in \mathcal{U}_{2}$.
Set $\mathcal{U}_{3} = \mathcal{U}_{2} \cap \mathcal{U}_{0}$.

\begin{lem} \label{isom Ck}
For any $\gamma \in \mathcal{U}_{3} \cap V_{1}^{k,a}$, 
\begin{align*}
(d F^{1,a})_{\gamma}|_{Y_{1}^{k,a}} 
:
Y_{1}^{1,a} \cap V_{1}^{k,a} = 
Y_{1}^{k,a} \rightarrow 
D_{1}^{k,a} (V_{1}^{k,a}) = Z_{2}^{0,a} \cap V_{2}^{k-1,a}
\end{align*}
is an isomorphism. 
\end{lem}

\begin{proof}
The injectivity of $(d F^{1,a})_{\gamma}|_{Y_{1}^{k,a}}$ 
follows from the fact that 
$(d F^{1,a})_{\gamma}|_{Y_{1}^{k,a}}$ 
is a restriction of 
the isomorphism 
$(d F^{1,a})_{\gamma}|_{Y_{1}^{1,a}} 
:Y_{1}^{1,a} \rightarrow Z_{2}^{0,a}$. 
The equation 
$(d F^{1,a})_{\gamma}|_{V_{1}^{k,a}} = (d F^{k,a})_{\gamma}$ 
and the smoothness of 
$F^{k,a}: V_{1}^{k,a} \rightarrow D_{1}^{k,a} (V_{1}^{k,a})$ 
imply that 
$(d F^{1,a})_{\gamma}|_{Y_{1}^{k,a}}$ is continuous. 

We prove that 
$(d F^{1,a})_{\gamma}|_{Y_{1}^{k,a}}$ is surjective. 
Take any $\mu \in D_{1}^{k,a} (V_{1}^{k,a})$. 
Since 
$(d F^{1,a})_{\gamma}|_{Y_{1}^{1,a}} 
:Y_{1}^{1,a} \rightarrow Z_{2}^{0,a}$
is an isomorphism, 
there exists $\beta \in G({\rm Im} (D_{1}^{1,a})) \subset V_{2}^{2,a}$ 
satisfying 
$(d F^{1,a})_{\gamma} ((D_{1}^{*})^{2,a} \beta) = \mu$.
Now we have 
\begin{align*}
(d F^{1,a})_{\gamma} ((D_{1}^{*})^{2,a} \beta)
=&
\left.
\frac{d}{dt} F^{1,a} (\gamma + t (D_{1}^{*})^{2,a} \beta ) \right|_{t=0}\\
=&
\left.
\frac{d}{dt} \left( F^{1,a}(\gamma + t (D_{1}^{*})^{2,a} \beta) 
+ (D_{2}^{*})^{1,a} D_{2}^{2,a} (t \beta)  
\right) \right|_{t=0}\\
=&
(d F^{1,a}_{\gamma})_{0} (\beta)
\end{align*}
by (\ref{eq Im D1}). 
Since $\gamma \in \mathcal{U}_{0} \cap V_{1}^{k,a}$, 
the differential operator $(d F^{1,a}_{\gamma})_{0}$ is 
the elliptic operator of class $C^{k-1, a}$. 
Hence by Schauder theory, 
$\beta$ is of class $C^{k+1,a}$, 
which implies that 
$\mu = (d F^{1,a})_{\gamma} ((D_{1}^{*})^{2,a} \beta) \in (d F^{1,a})_{\gamma} (Y_{1}^{k,a})$.
\end{proof}

Define the map $\tilde{G}^{1,a}: A_{1}^{1,a} \rightarrow V_{1}^{1,a} = X_{1}^{1,a} \oplus Y_{1}^{1,a}$ 
by $\tilde{G}^{1,a}(\alpha) = \alpha + G^{1,a} (\alpha)$. Set
\begin{align*}
\mathcal{U}_{4} = (\tilde{G}^{1,a})^{-1}(\mathcal{U}_{3}) \cap \mathcal{U}_{1}, 
\end{align*}
which is an open set of $A_{1}^{1,a}$.

\begin{lem} \label{smoothness G k,a}
For any $k \geq 1$, 
$G^{k,a} |_{\mathcal{U}_{4} \cap V_{1}^{k,a}}: \mathcal{U}_{4} \cap V_{1}^{k,a} 
\rightarrow Y_{1}^{k,a}$ 
is smooth. 
\end{lem}

\begin{proof}
We only have to prove that 
$G^{k,a}$ is smooth around any $\alpha_{0} \in \mathcal{U}_{4} \cap V_{1}^{k,a}$. 
Set $\gamma_{0} = \tilde{G}^{1,a}(\alpha_{0}) = \alpha_{0} + G^{1,a} (\alpha_{0})$. 
By Lemma \ref{image Ck}, $\gamma_{0} \in \mathcal{U}_{3} \cap V_{1}^{k,a}$. 
By Lemma \ref{isom Ck}, 
$
(d F^{1,a})_{\gamma_{0}}|_{Y_{1}^{k,a}} 
:
Y_{1}^{k,a} \rightarrow 
D_{1}^{k,a} (V_{1}^{k,a})
$
is an isomorphism. 
Set $\tilde{X}_{1}^{k,a} = \ker (d F^{1,a})_{\gamma_{0}} \cap V_{1}^{k,a}$. 
Then we have 
\begin{align*}
F^{1,a}(\gamma_{0}) =0, \qquad
V_{1}^{k,a} = \tilde{X}_{1}^{k,a} \oplus Y_{1}^{k,a}.
\end{align*}
Let $\tilde{\pi}: 
V_{1}^{k,a} = \tilde{X}_{1}^{k,a} \oplus Y_{1}^{k,a} \rightarrow \tilde{X}_{1}^{k,a}$ 
be the canonical projection 
and set $\tilde{\alpha}_{0} = \tilde{\pi} (\alpha_{0}).$
This is a smooth mapping between Banach spaces. 
Applying the implicit function theorem 
to $F^{k,a} = F^{1,a}|_{V_{1}^{k,a}}: V_{1}^{k,a} \rightarrow D_{1}^{k,a} (V_{1}^{k,a})$, 
there exist 
an open neighborhood $\tilde{U}_{1}^{k,a} \subset \tilde{X}_{1}^{k,a}$ of $\tilde{\alpha}_{0}$, 
an open set $\tilde{V}_{1}^{k,a} \subset Y_{1}^{k,a}$ 
and a smooth map 
$H^{k,a}: \tilde{U}_{1}^{k,a} \rightarrow \tilde{V}_{1}^{k,a}$ 
such that 
\begin{align*}
(F^{k,a})^{-1}(0) \cap (\tilde{U}_{1}^{k,a} \oplus \tilde{V}_{1}^{k,a})
=
\{ \tilde{\alpha} + H^{k,a}(\tilde{\alpha}); \tilde{\alpha} \in \tilde{U}_{1}^{k,a} \}.
\end{align*}
Now recall that 
for any 
$\alpha \in A_{1}^{1,a} \cap (\tilde{U}_{1}^{k,a} \oplus \tilde{V}_{1}^{k,a})$, 
we have 
$F^{1, a}(\alpha + G^{1, a}(\alpha)) = 0$ 
and 
$G^{k, a}(\alpha) = G^{1, a}(\alpha) \in Y_{1}^{k, a}$ by Lemma \ref{image Ck}. 
Then there exists $\tilde{\alpha} \in \tilde{U}_{1}^{k,a}$
satisfying $\alpha + G^{k, a}(\alpha) = \tilde{\alpha} + H^{k,a}(\tilde{\alpha})$. 
Taking $\tilde{\pi}$ of both sides, we obtain 
$\tilde{\pi}(\alpha) = \tilde{\alpha}$, which implies that 
\begin{align*}
G^{k,a}(\alpha) = \tilde{\pi} (\alpha) + H^{k,a} (\tilde{\pi}(\alpha)) - \alpha.
\end{align*}
Thus 
$G^{k,a}|_{\tilde{U}_{1}^{k,a} \oplus \tilde{V}_{1}^{k,a}}$ is smooth. 
\end{proof}

By Lemma \ref{smoothness G k,a}, 
it follows that 
$\hat{G}^{1,a} |_{\mathcal{U}_{4} \cap V_{1}}$ is smooth 
in the sense of Fr\'{e}chet, 
which implies that 
$\mathcal{M}_{L} (\Phi)$ is smooth around $[\iota]$.


Next, we prove that $\mathcal{M}_{L} (\Phi)$ is a submanifold of $\mathcal{M}(L,M)$ 
around $[\iota]$.
Set $\mathfrak{U} = (\mathcal{U}_{4} \cap V_{1}) \oplus Y_{1}$, 
where $Y_{1} = Y_{1}^{1,a} \cap V_{1}$.
Setting $X_{1} = X_{1}^{1,a} \cap V_{1}$, 
we have $V_{1} = X_{1}  \oplus Y_{1}$.
Let $p: V_{1} = X_{1}  \oplus Y_{1} \rightarrow X_{1}$
be the canonical projection.
Define the map $\psi: \mathfrak{U} \rightarrow \mathfrak{U}$ by 
$\psi (z) = z - \hat{G}^{1,a} \circ p (z)$.
By Lemma \ref{image Ck}, the image of $\psi$ is contained in $\mathfrak{U}$. 
This is bijective 
and the inverse $\psi^{-1}$ is given by 
$\psi^{-1}(z) = z + \hat{G}^{1,a} \circ p (z)$. 
Both mappings are smooth in the sense of Fr\'{e}chet. 
It is clear that 
\begin{align*}
\psi \left(\{ \alpha + \hat{G}^{1,a} (\alpha); \alpha \in \mathcal{U}_{4} \cap X_{1} \} \right)
=
\mathcal{U}_{4} \cap X_{1}
\end{align*}
since 
$\psi (\alpha + \hat{G}^{1,a} (\alpha)) = \alpha + \hat{G}^{1,a} (\alpha) - \hat{G}^{1,a} (\alpha) = \alpha$. 
Thus $\mathcal{M}_{L}(\Phi)$ is locally identified 
with the closed subspace 
$X_{1}$ of $V_{1}$, 
which implies that the
$\mathcal{M}_{L} (\Phi)$ is a submanifold of $\mathcal{M}(L,M)$ around $[\iota]$.

Finally, we prove the case 2. 
Since $P_{1} = D_{1}^{*} D_{1}$ is elliptic, 
we have $\dim \ker (D_{1}) \leq \dim \ker P_{1} < \infty$. 
Then we see the statement from the case 1. 
\end{proof}

Since the affine Legendrian condition is an open condition, 
we see as in the proof of \cite[Theorem 3.4]{Opozda} that 
the moduli space of affine Legendrian submanifolds, 
namely,
$\{ [\iota] \in \mathcal{M}(L,M); \iota \mbox{ is affine Legendrian} \}$, 
is a smooth Fr\'{e}chet manifold 
and it is open in $\mathcal{M}(L,M)$. 
Applying Proposition \ref{general moduli}, 
we prove Theorem \ref{smooth moduli affine Leg}.

\begin{proof}[Proof of Theorem \ref{smooth moduli affine Leg}]
Use the notation after Definition \ref{def of phi embed}. 
The moduli space of special affine Legendrian embeddings of $L$ 
is given by $\mathcal{M}_{L}({\rm Im} \psi).$
Fix any $[\iota] \in \mathcal{M}_{L}({\rm Im} \psi).$
Set 
$\mathcal{N}_{\iota} = \phi \iota_{*} TL \oplus \mathbb{R} \xi \circ \iota$. 
Define the map $F: U \rightarrow C^{\infty}(L)$ by
\begin{align*}
F(v) = * (\exp_{v}^{*} ({\rm Im} \psi)),
\end{align*}
where $*$ is the Hodge star operator of $\iota^{*}g$. 
Then the linearization $(dF)_{0}$ of $F$ at $0$ is given by
\begin{align*} 
(dF)_{0}(v) = * \iota^{*} L_{v} {\rm Im} \psi
=
* (\iota^{*} ( i(v)d {\rm Im} \psi + d i(v)  {\rm Im} \psi )).
\end{align*}
By \cite[Proposition 3.2]{Moriyama}, we have 
$d \psi = - (n+1) i \eta \wedge \psi$. 
Since $\iota$ is special affine Legendrian, we have $\iota^{*} {\rm Re}(\psi) = {\rm vol}_{\phi}[\iota]$. 
Then we compute 
\begin{align*}
\iota^{*} ( i(v)d {\rm Im} \psi
=
(n+1) (- \eta(v) {\rm vol}_{\phi}[\iota] 
+ \iota^{*} (\eta \wedge i(v) {\rm Re} \psi)).
\end{align*}
Denoting $v = \phi \iota_{*} Y + f \xi$ where $Y \in \mathfrak{X}(L), f \in C^{\infty}(L)$, we have 
\begin{align*}
i(v) \psi
&=
i (\phi \iota_{*} Y) i \left(r \frac{\partial}{\partial r} \right) \Omega |_{r=1}\\
&=
i (J \iota_{*} Y) i \left(r \frac{\partial}{\partial r} \right) \Omega |_{r=1}\\
&=
i \cdot i (\iota_{*} Y) i \left(r \frac{\partial}{\partial r} \right) \Omega |_{r=1}
=
i \cdot i (\iota_{*} Y) \psi.
\end{align*}
which implies that 
\begin{align*}
\iota^{*} (i(v) {\rm Re} \psi) = 0, \qquad
\iota^{*} d (i(v) {\rm Im} \psi) = d (i(Y) {\rm vol}_{\phi}[\iota]).
\end{align*}
Then we obtain 
\begin{align*}
D_{1}(v) = (dF)_{0}(v) 
&= * (-(n+1)* (\rho_{\phi}[\iota] f) + d * (\iota^{*} g (\rho_{\phi}[\iota]Y, \cdot)))\\
&=
-(n+1) \rho_{\phi}[\iota] f - d^{*} (\iota^{*} g (\rho_{\phi}[\iota]Y, \cdot)).
\end{align*}
Via the identification 
\begin{align*}
\begin{array}{cccc}
C^{\infty}(L, \mathcal{N}_{\iota}) =&
C^{\infty}(L, \phi \iota_{*} TL \oplus \mathbb{R} \xi \circ \iota) & \longrightarrow & 
C^{\infty}(L) \oplus \Omega^{1}(L) \\
&\rotatebox{90}{$\in$} & & \rotatebox{90}{$\in$} \\
&\phi \iota_{*}Y + f \xi & \longmapsto & (\rho_{\phi}[\iota]f, \iota^{*} g (\rho_{\phi}[\iota]Y, \cdot)), 
\end{array}
\end{align*}
the map $D_{1}$ is given by
\begin{align*}
C^{\infty}(L) \oplus \Omega^{1}(L) \ni (g, \alpha) \mapsto 
-(n+1) g - d^{*} \alpha 
\in C^{\infty}(L).
\end{align*}
Since $D_{1}^{*}(h) = (-(n+1)h, -dh)$, we see that 
\begin{align*}
D_{1} D_{1}^{*}(h) = (n+1)^{2} h + d^{*} d h,
\end{align*} 
which is clearly elliptic.
We easily see that ${\rm Im}(F) \subset {\rm Im}(D_{1})$ since ${\rm Im}(D_{1}) = C^{\infty}(L)$. 
Then 
setting $D_{2}=0$ and $V_{3} = \{ 0 \}$, we can apply Proposition \ref{general moduli} 
to see that   
the moduli space 
of special affine Legendrian embeddings of $L$ 
is an infinite dimensional smooth Fr\'{e}chet manifold 
modeled on the Fr\'{e}chet vector space 
$\{ (g, \alpha) \in C^{\infty}(L) \oplus \Omega^{1}(L); (n+1)g + d^{*} \alpha = 0 \} \cong \Omega^{1}(L)$. 
Note that 
we have 
$\Omega^{1}(L) = \{ \alpha \in \Omega^{1}(L); d^{*} \alpha = 0 \} \oplus d C^{\infty}(L)$ 
by the Hodge decomposition 
and $d C^{\infty}(L) = C^{\infty}(L) / \mathbb{R}$ is identified with 
the space of functions with integral $0$.

Since the moduli space of affine Legendrian submanifolds is open in $\mathcal{M}(L,M)$ 
and special affine Legendrian submanifolds are affine Legendrian, 
the proof is done.
\end{proof}

\begin{rem}
Applying Proposition \ref{general moduli} to the affine Lagrangian case, 
we can also deduce \cite[Theorem 1.1]{Opozda}.
\end{rem}


\begin{thebibliography}{99}
\bibitem{Baier}
P. D. Baier, Special Lagrangian Geometry, PhD thesis, University of Oxford, (2001).

\bibitem{Borrelli}
V. Borrelli, Maslov form and $J$-volume of totally real immersions, J. Geom.
Phys. 25 (1998), 271-290.

\bibitem{BoyerGalicki}
C. P. Boyer and K. Galicki, Sasakian geometry, Oxford Mathematical
Monographs, Oxford University Press, (2008). 

\bibitem{Chen}
B.-Y. Chen, Geometry of submanifolds and its applications, Science University of Tokyo, (1980).

\bibitem{Hamilton}
R. Hamilton, The inverse function theorem of Nash and Moser, Bull. Amer. Math. Soc. 7 (1982), 65-222.

\bibitem{HarveyLawson}
R. Harvey and H. B. Lawson, Calibrated geometries, Acta Math. 148 (1982), 47-157.

\bibitem{Kajigaya}
T. Kajigaya, Second variation formula and the stability of Legendrian minimal submanifolds in Sasakian manifolds,  Tohoku Math. J. (2) 65 (2013), 523-543.

\bibitem{LotayPacini}
J. D. Lotay and T. Pacini, Coupled flows, convexity and calibrations: Lagrangian and totally real geometry,  arXiv:1404.4227.

\bibitem{Moriyama}
T. Moriyama, Deformations of special Legendrian submanifolds in Sasaki-Einstein manifolds, 
arXiv:1306.2764. 

\bibitem{Ono}
H. Ono, Second variation and Legendrian stabilities of minimal Legendrian submanifolds in Sasakian manifolds, Differential Geom. Appl. 22 (2005), 327-340. 

\bibitem{Opozda}
B. Opozda, A moduli space of minimal affine Lagrangian submanifolds, 
Ann. Global Anal. Geom. 41 (2012), 535-547. 
\end{thebibliography}
\end{document}